\documentclass{amsart}
\usepackage{amsfonts}
\usepackage{bbm}
\usepackage{mathrsfs}
\usepackage{amsmath}
\usepackage{bigdelim}
\usepackage{multirow}
\usepackage{amssymb}
\usepackage{indentfirst}
\usepackage{CJK}
\usepackage{fancyhdr}
\usepackage{amscd,amssymb,amsmath,graphicx,verbatim,xy}
\usepackage[TS1,OT1,T1]{fontenc}

\newtheorem{theorem}{Theorem}[section]
\newtheorem{lemma}[theorem]{Lemma}
\newtheorem{corollary}[theorem]{Corollary}
\newtheorem{proposition}[theorem]{Proposition}

\theoremstyle{definition}
\newtheorem{definition}[theorem]{Definition}
\newtheorem{example}[theorem]{Example}

\newtheorem{question}[theorem]{Question}
\theoremstyle{remark}
\newtheorem{remark}[theorem]{Remark}
\numberwithin{equation}{section}

\begin{document}

\title[The Jordan decomposition and Kaplansky's second test problem]{The Jordan decomposition and Kaplansky's second test problem for Hermitian holomorphic vector bundles}
	
\author{Bingzhe Hou}
\address{Bingzhe Hou, School of Mathematics, Jilin University, 130012, Changchun, P. R. China}
\email{houbz@jlu.edu.cn}
	
\author{Chunlan Jiang}
\address{Chunlan Jiang, School of Mathematics, Hebei Normal University, 050016, Shijiazhuang, P. R. China}
\email{cljiang@hebtu.edu.cn}

\date{}
\thanks{}
\begin{abstract}
In 1954, I. Kaplansky proposed three test problems for deciding the strength of structural understanding of a class of mathematical objects in his treatise "Infinite abelian groups", which can be formulated for very general mathematical systems. In this paper, we focus on Kaplansky's second test problem in a context of complex geometry. Let $H^2_{\beta}$ be a weighted Hardy space and let $\textrm{Hol}(\overline{\mathbb{D}})$ be the set of all analytic functions on the unit closed disk $\overline{\mathbb{D}}$. The Cowen-Douglas operator theory tells us that each $h\in\textrm{Hol}(\overline{\mathbb{D}})$ induces a Hermitian holomorphic vector bundle on $H^2_{\beta}$, denoted by $E_{h(S_\beta)}(\Omega)$ ($E_{h(S_\beta)}$ in brief), where $\Omega$ is some certain connected open subset in $\mathbb{C}$. We show that the vector bundle $E_{h(S_\beta)}$ is a push-forwards Hermitian holomorphic vector bundle and study the similarity deformation problems of those bundles. Our main theorem is that if $H^2_{\beta}$ is a weighted Hardy space of polynomial growth, then for any $f\in \textrm{Hol}(\overline{\mathbb{D}})$, there exist a unique positive integer $m$ and an analytic function $h\in\textrm{Hol}(\overline{\mathbb{D}})$ inducing an indecomposable Hermitian holomorphic vector bundle $E_{h(S_{\beta})}$, such that $E_{f(S_\beta)}$ is similar to $\bigoplus_1^m E_{h(S_\beta)}$, where $h$ is unique in the sense of analytic automorphism group action. That could be seemed as a Jordan decomposition theorem for the push-forwards Hermitian holomorphic vector bundles. Furthermore, we give the similarity classification of those push-forwards Hermitian holomorphic vector bundles induced by analytic functions, and give an affirmative answer to Kaplansky's second test problem for those objects. As applications, we also give an affirmative answer to the geometric version and generalized version of a problem proposed by R. Douglas in 2007, and obtain the $K_0$-group of the commutant algebra of a multiplication operator on a weighted Hardy space of polynomial growth. In addition, we give an example to show the setting of polynomial growth condition is necessary, i.e., the related conclusions are possible to be false for the weighted Hardy spaces of intermediate growth.
\end{abstract}
\maketitle
\tableofcontents

\section{Introduction}

An important and standard problem in mathematics is to classify objects in a category up to some certain equivalence. In 1954, I. Kaplansky proposed three test problems for deciding the strength of structural understanding of a class of mathematical objects in his treatise "Infinite abelian groups" \cite{Kap54}. In particular, the second test problem is that if $G$ and $H$ are abelian groups and $G\oplus G$ is isomorphic to $H\oplus H$, is
$G$ isomorphic to $H$? In fact, the questions are very natural. If a direct sum operation is workable, as noted by Kaplansky himself, analogues of these problems "can be formulated for very general mathematical systems".  In the context of operator theory, Kaplansky's second test problem is that if $A$ and $B$ are bounded linear operators acting on an (infinite-dimensional) separable Hilbert space $\mathcal{H}$, and $A\oplus A$ is equivalent to $B\oplus B$ in a precise sense, is $A$ equivalent to $B$? In 1957,  R. Kadison and I. Singer \cite{Kad57} first considered the three test problems for single operators with respect to unitary equivalence. Jordan standard form plays an important role in the study of matrices and operators. In 1990, K. Davidson and D. Herrero \cite{Dav90} studied the Jordan form of a bitriangular operator and then answered the Kaplansky's test problems for the bitriangular operators with respect to quasi-similarity.
In 1997, E. Azoff \cite{Az95} also studied the three test problems in operator algebraic context including unitary equivalence of von Neumann algebras and similarity equivalence of representations of (non self-adjoint) matrix algebras. Although classifications for some specific classes have been obtained, the general Kaplansky's test problems remain open so far. In the present paper, we are interested in the Kaplansky's second test problem in the context of geometry. Hermitian holomorphic vector bundles over an open set $\Omega$ in the complex plane $\mathbb{C}$ naturally admit a direct sum operation. Then, Kaplansky's second test problem could be expressed as follows. If $E_1$ and $E_2$ are Hermitian holomorphic vector bundles over a domain $\Omega$ (a connected open set in $\mathbb{C}$) and $E_1\oplus E_1$ is equivalent to $E_2\oplus E_2$ in the sense of similarity, is $E_1$ equivalent to $E_2$ in the sense of similarity? In this paper, we focus on a class of Hermitian holomorphic vector bundles induced by analytic functions on the unit closed disk. We aim to provide the "Jordan form" of these Hermitian holomorphic vector bundles, and then answer the Kaplansky's second test problem in this geometric context. Moreover, we could give a similarity classification for these Hermitian holomorphic vector bundles induced by analytic functions.

Following from the isometric embedding theorem of Calabi \cite{Cal53}, every Hermitian holomorphic vector bundles over a domain $\Omega$ could be seemed as a sub-bundles in $\Omega\times \mathcal{H}$, where $\mathcal{H}$ is a complex separable Hilbert space. Then, we call two Hermitian holomorphic vector bundles similarity equivalent, if there exists a bounded invertible operator on $\mathcal{H}$ which transforms one bundle to the other one. In particular, M. Cowen and R. Douglas \cite{CD} introduced the eigenvector bundles from a class of geometric operators. Then, the similarity classification of Hermitian holomorphic vector bundles is closely related to the similarity classification of Cowen-Douglas operators. First of all, let us introduce and review some basic concepts related to holomorphic vector bundles (we refer to \cite{H05}, \cite{JJ07} and \cite{CD} to see more details).

Let $\mathcal{H}$ be a complex separable Hilbert space and $\mathcal{L}(\mathcal{H})$ be the set of all bounded linear operators from $\mathcal{H}$ to itself. Let $\mathcal{G}r(n,\mathcal{H})$ be the $n$-dimensional Grassmann manifold which is the set of all $n$-dimensional subspaces of $\mathcal{H}$. For a domain $\Omega$ in $\mathbb{C}$, we say that a map $\mathfrak{g}: \Omega\rightarrow\mathcal{G}r(n,\mathcal{H})$ is a holomorphic curve on $\Omega$, if there exist $n$ holomorphic $\mathcal{H}$-valued functions $\gamma_1, \gamma_2, \ldots, \gamma_n$ on $\Omega$ such that $\mathfrak{g}(\lambda)=\textrm{span}\{\gamma_1(\lambda), \gamma_2(\lambda), \ldots, \gamma_n(\lambda)\}$ is an $n$-dimensional subspace for any $\lambda\in\Omega$. Furthermore, a holomorphic curve $\mathfrak{g}$ naturally induces a Hermitian holomorphic vector bundle $E(\Omega)$ over $\Omega$, defined by
\[
E(\Omega)=\{(\lambda,\mathbf{x}); \ \lambda\in \Omega \ \text{and} \ \mathbf{x}\in \mathfrak{g}(\lambda)\}.
\]
We say that $\Omega$ is the base space of the bundle $E(\Omega)$, $\pi:E(\Omega)\rightarrow \Omega$ is the natural projection defined by $\pi(\lambda,\mathbf{x})=\lambda$, and $E(\lambda)\triangleq\pi^{-1}(\lambda)=\mathfrak{g}(\lambda)$ is the fibre of the bundle $E(\Omega)$ at $\lambda\in\Omega$. Given arbitrary $\lambda_0\in\Omega$. If $\gamma:\Delta \rightarrow\mathcal{H}$ is an $\mathcal{H}$-valued holomorphic function on some certain neighborhood $\Delta$ of $\lambda_0$ satisfying $0\neq\gamma(\lambda)\in\mathfrak{g}(\lambda)$ for any $\lambda\in \Delta$, we say $\gamma$ is a holomorphic cross-section of the bundle $E(\Omega)$ over $\lambda_0$. Moreover, if $\Gamma$ is a collection of $n$ holomorphic $\mathcal{H}$-valued functions $\gamma_1, \gamma_2, \ldots, \gamma_n$ on $\Delta$ such that $\Gamma(\lambda)=\{\gamma_1(\lambda), \gamma_2(\lambda), \ldots, \gamma_n(\lambda)\}$ is a sequence of linear independent vectors and spans the fibre $E(\lambda)$ for any $\lambda\in \Delta$, we call $\Gamma$ a holomorphic cross-section frame of the bundle $E(\Omega)$ over $\lambda_0$. Obviously, the vector bundle $E(\Omega)$ is Hermitian holomorphic if and only if for any $\lambda_0\in\Omega$, there exists a holomorphic frame of the bundle $E(\Omega)$ over $\lambda_0$. If it does not cause confusion, we could use $E$ instead of $E(\Omega)$ in brief.

Let $E(\Omega)$ and $\widetilde{E}(\Omega)$ be two Hermitian holomorphic vector bundles over a domain $\Omega$ induced by two holomorphic curves $\mathfrak{g}: \Omega\rightarrow\mathcal{G}r(n,\mathcal{H})$ and $\widetilde{\mathfrak{g}}: \Omega\rightarrow\mathcal{G}r(\widetilde{n},\widetilde{\mathcal{H}})$, respectively. The direct sum of the two vector bundles $E(\Omega)$ and $\widetilde{E}(\Omega)$ is defined by
\[
E(\Omega) \oplus \widetilde{E}(\Omega)=\{(\lambda,\mathbf{v}); \ \lambda\in \Omega \ \text{and} \ \mathbf{v}\in E(\lambda)\oplus \widetilde{E}(\lambda)\}.
\]
Notice that $E(\Omega)\bigoplus\widetilde{E}(\Omega)$ is a $(n+\widetilde{n})$-dimensional Hermitian holomorphic vector bundles over $\Omega$ induced by the holomorphic curve $\mathfrak{g}\oplus\widetilde{\mathfrak{g}}: \Omega\rightarrow\mathcal{G}r(n+\widetilde{n},\mathcal{H}\oplus\widetilde{\mathcal{H}})$.

For a bounded linear operator $X:\mathcal{H}\rightarrow\widetilde{\mathcal{H}}$, denote
\[
XE(\Omega)=\{(\lambda,X(\mathbf{x})); \ \lambda\in \Omega \ \text{and} \ \mathbf{x}\in E(\lambda)\}.
\]
If there is a bounded invertible linear operator $X$ such that $XE(\Omega)=\widetilde{E}(\Omega)$, we say that the two Hermitian holomorphic vector bundles $E(\Omega)$ and $\widetilde{E}(\Omega)$ are similar and $X$ is a similarity deformation from $E(\Omega)$ to $\widetilde{E}(\Omega)$. In particular, if the above $X$ is a unitary operator, we say that $E(\Omega)$ and $\widetilde{E}(\Omega)$ are unitarily equivalent and $X$ is a unitary deformation. Obviously, if two Hermitian holomorphic vector bundles are similar, they must have the same dimension.

A bounded linear operator $T$ is said to be invariant on the Hermitian holomorphic vector bundle $E(\Omega)$, if for any $\lambda\in \Omega$,
\[
TE(\lambda)\subseteq E(\lambda).
\]
We say that the invariant deformation algebra of the vector bundle $E(\Omega)$, denoted by $\mathcal{L}_{inv}(E(\Omega))$, is the set of all invariant bounded linear operators on the Hermitian holomorphic vector bundle $E(\Omega)$. It is not difficult to see that $\mathcal{L}_{inv}(E(\Omega))$ is a Banach algebra. Moreover, if the idempotents in $\mathcal{L}_{inv}(E(\Omega))$ are just the trivial ones~---~the zero operator $\mathbf{0}$ and the identity operator $\mathbf{I}$, the vector bundle $E(\Omega)$ is said to be indecomposable, which means $E(\Omega)$ can not be similar to a direct sum of two non trivial vector bundles.

The vector bundles induced by mappings are usually studied. For instance, the pull-back bundle is a classical research object. In the present paper, we are interested in the bundles induced in a push-forwards way. Compared with the pull-back bundles, well-defined push-forwards bundles usually require more conditions, see \cite{Harts} and \cite{For} for examples. Now, we give the definition of push-forwards Hermitian holomorphic vector bundles as follows.

\begin{definition}\label{PFVB}
Let $\Lambda$ be a domain and $h$ be a  non-constant analytic function on $\Lambda$. Suppose that $\Omega$ is a domain in $h(\Lambda)$. Let $E(\Lambda)$ and $E(\Omega)$ be two Hermitian holomorphic vector bundles over $\Lambda$ and $\Omega$, respectively. If for each $\omega\in\Omega$,
\[
E(\omega)=\textrm{span}\{E(\lambda);~ \lambda\in h^{-1}(\omega) \},
\]
then $E(\Omega)$ is said to be the push-forwards Hermitian holomorphic vector bundles induced by $h$ from $E(\Lambda)$.
\end{definition}

The similarity classification of the Hermitian holomorphic vector bundles is naturally an important problem. Notice that similarity equivalence is much stronger than bundle isomorphism equivalence. As well known, every holomorphic vector bundle over a domain in $\mathbb{C}$ is trivial in the sense of bundle isomorphism. However, the similarity classification of the Hermitian holomorphic vector bundles is very complicated.

In \cite{CD}, M. Cowen and R. Douglas studied the eigenvector bundles of a class of geometric operators from a complex geometric perspective. This class of operators, named Cowen-Douglas operators now, connects operator theory and complex geometry.

\begin{definition}[\cite{CD}]
For $\Omega$ a connected open subset of $\mathbb{C}$ and $n$ a
positive integer, let $\mathbf {B}_{n}(\Omega)$ denote the set of the
operators $T$ in $\mathcal{L}(\mathcal{H})$ which satisfy:
\begin{enumerate}
\item[(a)] $\Omega \subseteq \sigma(T)=\{\omega \in \mathbb{C}: \ T-\omega\mathbf{I}
 \ \text{not \ invertible}\}$;
\item[(b)] ${\textrm Ran}(T-\omega\mathbf{I})=\mathcal{H}$ for every $\omega$ in $\Omega$;
\item[(c)] $\textrm{Span}\{{\textrm Ker}(T-\omega\mathbf{I}); {\omega\in \Omega}\}=\mathcal{H}$, where $\textrm{Span}\{\cdot\}$ denotes the closed linear span of some certain subset of $\mathcal{H}$;
\item[(d)] $\dim {\textrm Ker}(T-\omega\mathbf{I})=n$ for every $\omega$ in $\Omega$.
\end{enumerate}
\end{definition}

The backward (weighted) shift operator is a fundamental and important example of Cowen-Douglas operator. More precisely, let $H^2_{\beta}$ be a weighted Hardy space with $\beta_{n+1}/\beta_n\rightarrow 1$ and
let $S_{\beta}$ be the standard left inverse of the multiplication operator $M_z$ on $H^2_{\beta}$, defined by
\[
S_{\beta}(f(z))=\frac{f(z)-f(0)}{z}, \ \ \ \text{for every} \ f(z)\in H^2_{\beta}.
\]
Notice that $S_{\beta}$ is the backward shift on $H^2_{\beta}$ under the orthogonal base $\{z^k\}_{k=0}^{\infty}$ and $S_{\beta}\in\mathbf {B}_{1}(\mathbb{D})$. Then, each $S_{\beta}$ induces a Hermitian holomorphic vector bundle $E_{S_{\beta}}(\mathbb{D})$,
\[
E_{S_{\beta}}(\mathbb{D})=\{(\lambda,\mathbf{x}); \ \lambda\in \mathbb{D} \ \text{and} \ \mathbf{x}\in \textrm{Ker}(S_{\beta}-\lambda\mathbf{I})\}.
\]

Furthermore, each analytic function on the closed unit disk will naturally induce a Hermitian holomorphic vector bundle $E_{h(S_{\beta})}(\Omega)$, where $\Omega$ is a domain in $h(\mathbb{D})\setminus h(\mathbb{T})$,
\[
E_{h(S_{\beta})}(\Omega)=\{(\omega,\mathbf{x}); \ \omega\in \Omega \ \text{and} \ \mathbf{x}\in \textrm{Ker}(h(S_{\beta})- \omega\mathbf{I})\}.
\]

In 1978, Cowen and Douglas \cite{CD} gave the unitary classification of $\mathbf {B}_{n}(\Omega)$ by Riemann curvature, and they asked whether the ratio of the curvatures is the similarity invariant. In 1983, Clark and Misra \cite{CM83} gave a negative answer to this conjecture, and their counterexample is a backward unilateral shift operator. Up to now, the similarity classification of Hermitian holomorphic vector bundles and Cowen-Douglas operators attracts the attention of numerous researchers. For instance, Clark, Misra and Uchiyama studied the operators similar to backward shift operators in \cite{Mis}, \cite{CM85}, \cite{CM} and \cite{U90}. C. Jiang et. al. used the K-theory of Banach algebras to give the similarity classification of holomorphic curves and Cowen-Douglas operators in a series of articles \cite{J04}, \cite{JGJ}, \cite{Ji}, \cite{JL}, \cite{JJ07}. More recently, Jiang, Ji and Keshali \cite{JJK23} gave a geometric similarity invariant for a class of Cowen-Douglas operators in 2023. Another useful method comes from the techniques in solving the famous (operator value) Corona Problem. In particular, it plays an important role of the techniques in Harmonic analysis and measure theory in the article \cite{TW09} of Treil and Wick. Then, Kwon and Treil \cite{KT09} characterized the contractions similar to the backward shift on the Hardy space. Furthermore, Douglas, Kwon and Treil \cite{D13} characterized the $n$-hypercontraction backward shift on the vector Bergman space. However, the similarity classification of Hermitian holomorphic vector bundles or Cowen-Douglas operators remains problematic.

A natural topic is to study the similarity classification of the Hermitian holomorphic vector bundles $E_{h(S_{\beta})}$. It is certainly an important issue in this topic of the Kaplansky's second test problem for those Hermitian holomorphic vector bundle with respect to similarity equivalence. In addition, if the analytic function on the closed unit disk is an analytic automorphism $\varphi(z)$, is $E_{\varphi(S_{\beta})}$ similar to  $E_{S_{\beta}}$? If the analytic function on the closed unit disk is a finite Blaschke product $B(z)$ with order $m$, is $E_{B(S_{\beta})}$ similar to $\bigoplus_{j=1}^{m}E_{S_{\beta}}$? This is a geometric bundles version of a problem proposed by R. Douglas \cite{D07} in 2007. Notice that the above questions are related to the weights of weighted Hardy space $H^2_{\beta}$. For instance, if $H^2_{\beta}$ is just the classical Hardy space $H^2$, it is well-known that $E_{\varphi(S_{\beta})}$ is unitary to $E_{S_{\beta}}$. However, in general case, $E_{\varphi(S_{\beta})}$ is not unitary to $E_{S_{\beta}}$ and it is not clear whether they are similar.

In this paper, we answer the above questions. We give a Jordan decomposition theorem for the push-forwards Hermitian holomorphic vector bundles. Furthermore, we give the similarity classification of those push-forwards Hermitian holomorphic vector bundles induced by analytic functions, and give an affirmative answer to Kaplansky's second test problem for those objects. In Section 2, we state the main theory of this paper. We review the weighted Hardy spaces and propose the concept of polynomial growth, and then we show that the vector bundle $E_{h(S_{\beta})}$ is a push-forwards Hermitian holomorphic vector bundle and state the main results. In Section 3, we use base theory and K-theory of Banach algebra (invariant deformation algebra) to prove the main theorems. In the last section, we give some remarks and applications. More precisely, we discuss the case of intermediate growth, and as applications, we answer the problem proposed by R. Douglas in 2007 and obtain the $K_0$-group of the commutant algebra of a multiplication operator on a weighted Hardy space of polynomial growth.

\section{The statement of main theory}

\subsection{Weighted Hardy spaces}
Denote by $\textrm{Hol}(\mathbb{D})$ the space of all analytic functions on the unit open disk $\mathbb{D}$.
For any $f\in \textrm{Hol}(\mathbb{D})$, denote the Taylor expansion of $f(z)$ by
\[
f(z)=\sum\limits_{k=0}^{\infty}\widehat{f}(k)z^k.
\]
In this paper, we introduce the weighted Hardy space from a given weight sequence.
Let $\beta=\{\beta_k\}_{k=0}^{\infty}$ be a sequence of positive numbers with $\beta_0=1$. For $k=1, 2, \cdots$, write
\[
w_k=\frac{\beta_k}{\beta_{k-1}}.
\]
The weighted Hardy space $H^2_{\beta}$ induced by the weight sequence $\beta$ (or $w$) is defined by
\[
H^2_{\beta}=\{f(z)=\sum\limits_{k=0}^{\infty}\widehat{f}(k)z^k; \ \sum\limits_{k=0}^{\infty}|\widehat{f}(k)|^{2}{\beta}_k^2<\infty\}.
\]
Moreover, the weighted Hardy space $H^2_{\beta}$ is a complex separable Hilbert space, on which the inner product is defined by,  for any $f,g \in H^2_{\beta}$,
\[
\langle  f, g\rangle_{H^2_{\beta}}=\sum\limits_{k=0}^{\infty}{\beta}_k^2\overline{\widehat{g}(k)}\widehat{f}(k).
\]
Then, any $f\in H^2_{\beta}$ has the following norm
\[
\|f(z)\|_{H^2_{\beta}}=\sqrt{\langle  f, f\rangle_{H^2_{\beta}}}=\sqrt{\sum\limits_{k=0}^{\infty}{\beta}_k^2|\widehat{f}(k)|^{2}}.
\]
In particular, $\|z^n\|_{H^2_{\beta}}={\beta}_n$ for each $n\in \mathbb{\mathbb{N}}$. Let $H^2_{\beta}$ and $H^2_{\beta'}$ be two weighted Hardy spaces. If there are positive constants $K_1$ and $K_2$ such that $K_1\leq {\beta'_k}/{\beta_k}\leq K_2$ for all $k\in\mathbb{N}$, then $H^2_{\beta}=H^2_{\beta'}$ and their norms are equivalent, and hence we say that $H^2_{\beta}$ and $H^2_{\beta'}$ are equivalent weighted Hardy spaces.

Furthermore, any $f\in H^2_{\beta}$ induces a multiplication operator $M_f$, defined by
\[
M_f(g)=f\cdot g, \ \  \ \ \text{for \ any }\ g\in H^2_{\beta}.
\]
Define $H^{\infty}_{\beta}$ by the set
\[
H^{\infty}_{\beta}=\{f(z)\in H^2_{\beta}; \ M_f \ \text{is  a  bounded  operator  from} \ H^2_{\beta} \ \text{to} \ H^2_{\beta}\}.
\]
In the present article, we always assume the weight sequence $w$ satisfying
\[
\lim\limits_{k\rightarrow\infty}w_k=1.
\]
The above condition is natural. The classical Hardy space, the weighted Bergman spaces, and the weighted Dirichlet spaces are all the weighted Hardy spaces satisfying the above condition.
In this case, it is not difficult to see that
\[
\textrm{Hol}(\overline{\mathbb{D}})\subseteq H^{\infty}_{\beta}\subseteq H^2_{\beta} \subseteq \textrm{Hol}(\mathbb{D}),
\]
where $\textrm{Hol}(\overline{\mathbb{D}})$ is the space of all analytic functions on the unit closed disk $\overline{\mathbb{D}}$.

Following from the work of Shields \cite{Shi}, to study the multiplication operator $M_z$ on $H^{2}_{\beta}$ is equivalent to study the forward unilateral weighted shift $S^{for}_{w}$ on the classical Hardy space, where $S^{for}_{w}:H^2\rightarrow H^2$ is defined by
\[
S^{for}_{w}(z^k)=w_{k+1}z^{k+1}, \ \ \ \text{for} \ k=0,1,2,\ldots.
\]
Notice that $w_k\rightarrow 1$ implies the spectrum of $S^{for}_{w}$, denoted by $\sigma(S_w)$, is the unit closed disk $\overline{\mathbb{D}}$. Then, for any $f\in \textrm{Hol}(\overline{\mathbb{D}})$, one can see that the analytic function calculus $f(S_w)$ on $H^2$ is corresponding to the multiplication operator $M_f$ on $H^{2}_{\beta}$. Consequently, we have $\sigma(M_f)=f(\overline{\mathbb{D}})$. Therefore, $M_f$ is a bounded operator on $H^{2}_{\beta}$ and moreover, $M_f$ is lower bounded on $H^{2}_{\beta}$ if and only if $f$ has no zero point in $\partial\mathbb{D}$. In addition, $S^{for}_{w}$ is a right inverse of $S_{\beta}$.

Let $\mathcal{H}$ be a complex separable Hilbert space. An algebra $\mathcal{A}$ of operators on $\mathcal{H}$ is called strictly
cyclic if there exists a vector $f_0$ such that $\mathcal{A}f_0=\mathcal{H}$. Furthermore, a bounded linear operator $T$ is called strictly cyclic, if the weakly closed algebra generated by $T$ is strictly cyclic. In particular, the multiplication operator $M_z$ is strictly cyclic on a weighted Hardy space $H^{2}_{\beta}$ if and only if $H^{\infty}_{\beta}=H^{2}_{\beta}$ (Proposition 31 in \cite{Shi}).

Let $H^2_{\alpha}$ and $H^2_{\beta}$ be two weighted Hardy spaces. Define the linear operator $T_{\alpha,\beta}:~H^2_{\alpha}\rightarrow H^2_{\beta}$ by
\[
T_{\alpha,\beta}(f(z))=\sum\limits_{k=0}^{\infty}\frac{\widehat{f}(k)\alpha_k}{\beta_k}z^k, \ \ \ \ \ \text{for any} \ f(z)=\sum\limits_{k=0}^{\infty}\widehat{f}(k)z^k\in H^2_{\alpha}.
\]
Obviously, the operator $T_{\alpha,\beta}$ is an isometry from $H^2_{\alpha}$ to $H^2_{\beta}$.

The authors have introduced the growth types of weighted Hardy spaces in a previous paper \cite{HJ}.

\begin{definition}[\cite{HJ}]\label{growth}
Let $w=\{w_k\}_{k=1}^{\infty}$  be a sequence of positive numbers with $w_k\rightarrow 1$.
\begin{enumerate}
 \item \ If $\sup_{k}(k+1)|w_k-1|<\infty$, we say that the weighted Hardy space $H^2_{\beta}$ is of polynomial growth.
 \item \ If $\sup_{k}(k+1)|w_k-1|=\infty$, we say that the weighted Hardy space $H^2_{\beta}$ is of intermediate growth.
\end{enumerate}
\end{definition}

\begin{remark}\label{growthremark}
It is easy to see that the condition $\sup_{k}(k+1)|w_k-1|<\infty$ holds if and only if there exists a positive number $M$ such that for each $k\in\mathbb{N}$,
\[
\frac{k+1}{k+M+1}\leq w_k \leq \frac{k+M+1}{k+1}.
\]
Moreover, such weighted Hardy space $H^2_{\beta}$ is said to be of $M$-polynomial growth. Roughly speaking, the polynomial growth condition for a weighted Hardy space $H^2_{\beta}$ implies that $\beta_n$ is controlled by a monomial of $(n+1)$ as an upper bound and a monomial of $\frac{1}{n+1}$ as a lower bound.

Notice that the classical Hardy space, the weighted Bergman spaces, the weighted Dirichlet spaces and the Sobolev space are all the weighted Hardy spaces of polynomial growth.
\end{remark}

\noindent \textbf{Notation.} \ For any $f=\sum\limits^{\infty}_{k=0}\widehat{f}(k)z^k\in\textrm{Hol}(\mathbb{D})$, denote
\[
f^{*}(z)=\overline{f(\overline{z})}=\sum\limits^{\infty}_{k=0}\overline{\widehat{f}(k)}z^k.
\]
Obviously, $f(z)\in\textrm{Hol}(\mathbb{D})$ if and only if $f^{*}(z)\in\textrm{Hol}(\mathbb{D})$.

Denote by $\textrm{Aut}(\mathbb{D})$ the analytic automorphism group on $\mathbb{D}$, which is the set of all analytic bijections from $\mathbb{D}$ to itself. As well known, each $\varphi\in\textrm{Aut}(\mathbb{D})$ is a M\"{o}bius transformation which could be written as the following form
\[
\varphi(z)={\textrm e}^{\mathbf{i}\theta}\cdot \frac{z_0-z}{1-\overline{z_0}z}, \ \ \ \text{for some} \ \theta\in \mathbb{R} \ \text{and} \ z_0\in\mathbb{D}.
\]
Moreover, a Blaschke product $B(z)$ on $\mathbb{D}$ with order $m$ is a product of $m$ M\"{o}bius transformations, i.e.,
\[
B(z)=\textrm{e}^{\mathbf{i}\theta}\cdot\prod^m_{j=1}\frac{z_{j}-z}{1-\overline{z_{j}}z},
\]
where $\theta\in \mathbb{R}$ and $z_j\in\mathbb{D}$ for $j=1, \cdots, m$. Obviously, $B(z)$ is a Blaschke product with order $m$ (or a M\"{o}bius transformation) if and only if $B^{*}(z)$ is a Blaschke product with order $m$ (or a M\"{o}bius transformation).

\subsection{Push-forwards Hermitian holomorphic vector bundles}

Given any $h\in\textrm{Hol}(\overline{\mathbb{D}})$. Let $\Omega$ be a connected component of the open set $h(\mathbb{D})\setminus h(\mathbb{T})$. Obviously, $\Omega$ is a connected open subset in $\mathbb{C}$. Following from basic complex analysis, it is not difficult to see that for any $\omega\in\Omega$, the sum of the multiplicities of zeros of $h(\lambda)-\omega$ in $\mathbb{D}$ is the same integer, denoted by $n$. From a point view of the Riesz-Dunford functional calculus for bounded linear operators (see \cite{CHM} for example), one can see that the above $n$ is just the Fredholm index of $h(S_{\beta})- \omega\mathbf{I}$ for any $\omega\in\Omega$. A bounded operator $T$ on a Hilbert space $\mathcal{H}$ is said to be Fredholm, if $\textrm{dim}\textrm{Ker} ~T$ and $\textrm{dim}\textrm{Ker} ~T^*$ are finite. Moreover, the Fredholm index of $T$ is defined by
\[
\textrm{ind} ~T=\textrm{dim}\textrm{Ker} ~T-\textrm{dim}\textrm{Ker} ~T^*.
\]
For convenience, we say that the operator $h(S_{\beta})$ has index $n$ on the domain $\Omega$.

Notice that the open set $h(\mathbb{D})\setminus h(\mathbb{T})$ may have several connected components and the operator $h(S_{\beta})$ may have different indices on distinct components. For example, let $h(\lambda)=\lambda^2+\lambda+2$, then as shown in Figure \ref{figure1}, $h(S_{\beta})$ has index $1$ on the domain $\Omega_1$ (the red domain) and has index $2$ on the domain $\Omega_2$ (the yellow domain).
\begin{figure}[h]
		\centering
		\includegraphics[width=0.8\textwidth]{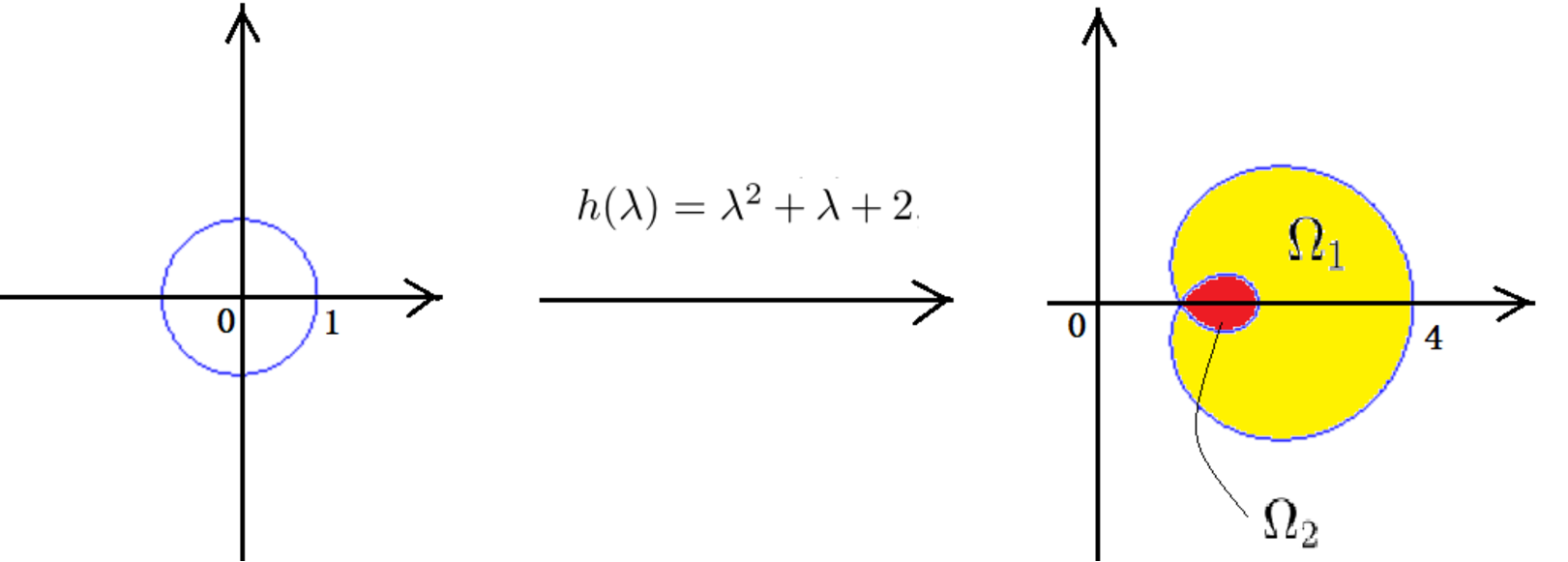}
		\caption{The image of $h(\lambda)=\lambda^2+\lambda+2$}
		\label{figure1}
\end{figure}	

From now on, let $h$ be an analytic function on $\overline{\mathbb{D}}$ and let $\Omega$ be a connected component in $h(\mathbb{D})\setminus h(\mathbb{T})$ with index $n$ which is the Fredholm index of $h(S_{\beta})$ on $\Omega$. For a point $\lambda\in\mathbb{D}$, if $h'(\lambda)=0$, we say that $\lambda$ is a branched point of $h$. Denote by $\mathcal{Z}(h')$ the set of all branched points of $h$. Moreover, let $\widetilde{\Omega}=\Omega\setminus h(\mathcal{Z}(h'))$. Since $\mathcal{Z}(h')$ is a finite set, $\widetilde{\Omega}$ is also a domain. Notice that for any $\omega\in\widetilde{\Omega}$, $h(\lambda)-\omega$ has just $n$ distinct zeros $\lambda_1, \ldots, \lambda_n$ in $\mathbb{D}$. Moreover, there is a neighborhood $V$ of $\omega$ and neighborhood $U_i$ of $\lambda_i$ for $i=1, \ldots, n$ such that $h:U_i\rightarrow V$ is bi-holomorphic. So, we say that $\widetilde{\Omega}$ is an unbranched domain.

Recall that
\[
E_{h(S_{\beta})}(\Omega)=\{(\omega,\mathbf{x}); \ \omega\in \Omega \ \text{and} \ \mathbf{x}\in \textrm{Ker}(h(S_{\beta})- \omega\mathbf{I})\}.
\]
One can see that the vector bundle $E_{h(S_{\beta})}(\Omega)$ is induced by $h$ from the vector bundle $E_{S_{\beta}}(\mathbb{D})$.

\begin{lemma}\label{Pre-Push-forwards}
Let $h\in\textrm{Hol}(\overline{\mathbb{D}})$ and $\Omega$ be a domain in $h(\mathbb{D})\setminus h(\mathbb{T})$. Suppose that the operator $h(S_{\beta})$ has index $n$ on the domain $\Omega$. For any given $\omega\in\Omega$, let $z_i$ be the zero points of $h(z)-\omega$ with multiplicity $n_i$ for $i=1,\ldots, p$, and $n_1+\ldots+n_{p}=n$. Then the fibre of $E_{h(S_\beta)}(\Omega)$ at $\omega$ is $\textrm{span}\{\textrm{Ker}(S_\beta- z_i\mathbf{I})^{n_i};~ i=1, \ldots, p.\}$.
\end{lemma}
\begin{proof}
Following from $h(z)-\omega\in\textrm{Hol}(\overline{\mathbb{D}})$, one can get its inner-outer factor decomposition
\[
h(z)-\omega=F(z)\prod\limits_{i=1}^{p}\left(\frac{z_i-z}{1-\overline{z_i}z} \right)^{n_i},
\]
where $F(z)\in (H^{\infty})^{-1}\bigcap\textrm{Hol}(\overline{\mathbb{D}})$ since $h(z)-\omega$ has no zeros on $\mathbb{T}$. Then,
\[
h(S_\beta)-\omega\mathbf{I}=F(S_\beta)\prod\limits_{i=1}^{p}({\mathbf{I}-\overline{z_i}S_\beta})^{-n_i}({z_i\mathbf{I}-S_\beta})^{n_i}.
\]
Since $F(S_\beta)\prod\limits_{i=1}^{p}({\mathbf{I}-\overline{z_i}S_\beta})^{-n_i}$ is a bounded invertible operator, we have
\[
E_{h(S_\beta)}(\omega)=\textrm{Ker}(h(S_{\beta})- \omega\mathbf{I})=\textrm{span}\{\textrm{Ker}(S_\beta- z_i\mathbf{I})^{n_i};~ i=1, \ldots, p.\}.
\]
\end{proof}
In particular, if $\Omega$ is an unbranched domain, the vector bundle $E_{h(S_{\beta})}(\widetilde{\Omega})$ is just a push-forwards vector bundle induced by $h$ from the vector bundle $E_{S_{\beta}}(\mathbb{D})$.

\begin{lemma}[\cite{JZ}]\label{mmatrix}
Let $z_1, \cdots, z_{m}$ be $m$ distinct points in $\mathbb{D}$.
Then the matrix
\[
\begin{bmatrix}
\frac{1}{1-|z_1|^2} & \frac{1}{1-\overline{z_1}z_2} & \frac{1}{1-\overline{z_1}z_3}  & \cdots & \frac{1}{1-\overline{z_1}z_m} \\
\frac{1}{1-\overline{z_2}z_1}   & \frac{1}{1-|z_2|^2} & \frac{1}{1-\overline{z_2}z_3}  & \cdots & \frac{1}{1-\overline{z_2}z_m}  \\
\frac{1}{1-\overline{z_3}z_1}  & \frac{1}{1-\overline{z_3}z_2}   & \frac{1}{1-|z_3|^2}  & \cdots  & \frac{1}{1-\overline{z_3}z_m} \\
\vdots   & \vdots & \vdots & \ddots &\vdots \\
\frac{1}{1-\overline{z_m}z_1}   & \frac{1}{1-\overline{z_m}z_2} & \frac{1}{1-\overline{z_m}z_3} & \cdots &\frac{1}{1-|z_m|^2} \\
\end{bmatrix}
\]
is invertible.
\end{lemma}

\begin{lemma}\label{Push-forwards}
Let $h\in\textrm{Hol}(\overline{\mathbb{D}})$ and $\Omega$ be a domain in $h(\mathbb{D})\setminus \left(h(\mathbb{T})\bigcup h(\mathcal{Z}(h'))\right)$. Suppose that the operator $h(S_{\beta})$ has index $n$ on the domain $\Omega$. Then for any $\omega\in\Omega$, the fibre $E_{h(S_{\beta})}(\omega)$ is an $n$-dimensional vector space and
\[
E_{h(S_{\beta})}(\omega)=\textrm{span}\{E_{S_{\beta}}(z_i);~z_i\in h^{-1}(\omega), \ i=1, \ldots, n.\}.
\]
\end{lemma}
\begin{proof}
Since $h(z)-\omega$ has $n$ distinct zero points in $\mathbb{D}$, it follows from Lemma \ref{Pre-Push-forwards} that
\begin{align*}
E_{h(S_\beta)}(\omega)&=\textrm{span}\{\textrm{Ker}(S_\beta- z_i\mathbf{I});~z_i\in h^{-1}(\omega), \ i=1, \ldots, n.\} \\
&=\textrm{span}\{E_{S_{\beta}}(z_i);~z_i\in h^{-1}(\omega), \ i=1, \ldots, n.\}.
\end{align*}
Notice that $E_{S_{\beta}}(z_i)=\{\frac{\lambda}{1-{z_i}z};~\lambda\in\mathbb{C}\}$ for $i=1, \ldots, n$. By Lemma \ref{mmatrix}, the vectors $\{\frac{1}{1-{z_1}z}, \cdots, \frac{1}{1-{z_n}z}\}$ are linear independent and then the fibre $E_{h(S_{\beta})}(\omega)$ is an $n$-dimensional vector space.
\end{proof}

\begin{lemma}\label{leftin}
Suppose that $H^2_{\beta}$ is the weighted Hardy space induced by a weight sequence $w=\{w_k\}_{k=1}^{\infty}$ with $w_k\rightarrow 1$. Let $B(z)=\prod^m_{j=1}\frac{z_{j}-z}{1-\overline{z_{j}}z}$ be a Blaschke product on $\mathbb{D}$ with order $m$.
Then $B^{*}(S_{\beta})\in \mathbf {B}_m(\mathbb{D})$ and $B^{*}(S_{\beta})B(M_z)={\textbf I}$.
\end{lemma}

\begin{proof}
Notice that $S_{\beta}\in\mathbf {B}_1(\mathbb{D})$ and $B^{*}(z)$ is also a Blaschke product on $\mathbb{D}$ with order $m$. Then it is easy to see $B^{*}(S_{\beta})\in \mathbf {B}_m(\mathbb{D})$.

For any $\alpha\in\mathbb{D}$, let $\varphi_{\alpha}(z)=\frac{\alpha-z}{1-\overline{\alpha}z}$. Then, it follows from $S_{\beta}M_z={\textbf I}$ that
\begin{align*}
\varphi_{\alpha}^{*}(S_{\beta})\varphi_{\alpha}(M_z)
&=(\overline{\alpha}-S_{\beta})\left(\sum\limits_{k=0}^{\infty}(\alpha S_{\beta})^k\right)(\alpha-M_z)({\textbf I}-\overline{\alpha}M_z)^{-1} \\
&=(\overline{\alpha}-S_{\beta})(-M_z)({\textbf I}-\overline{\alpha}M_z)^{-1} \\
&={\textbf I}.
\end{align*}
Since
$B(z)=\prod^m_{j=1}\varphi_{z_{j}}(z)$,
we have
\[
B^{*}(S_{\beta})B(M_z)=\prod\limits_{j=1}^{m}\varphi_{z_{j}}^{*}(S_{\beta})\prod\limits_{j=1}^{m}\varphi_{z_{j}}(M_z)={\textbf I}.
\]
\end{proof}

\begin{proposition}\label{PFHHVB}
Let $h\in\textrm{Hol}(\overline{\mathbb{D}})$ and $\Omega$ be a domain in $h(\mathbb{D})\setminus \left(h(\mathbb{T})\bigcup h(\mathcal{Z}(h'))\right)$. Then,
$E_{h(S_{\beta})}(\Omega)$ is a push-forwards Hermitian holomorphic vector bundle induced by $h$ from the vector bundle $E_{S_{\beta}}(\mathbb{D})$.
\end{proposition}

\begin{proof}
Suppose that $h(S_\beta)$ has index $n$ on the domain $\Omega$. For any $\omega_0\in\Omega$, $h(z)-\omega_0$ has $n$ distinct zero points $\overline{z_1}, \cdots , \overline{z_n}$ in $\mathbb{D}$. Then, by Lemma \ref{Pre-Push-forwards},
\[
h(S_\beta)-\omega_0\mathbf{I}=F(S_\beta)B^{*}(S_\beta),
\]
where $F(z)\in (H^{\infty})^{-1}\bigcap\textrm{Hol}(\overline{\mathbb{D}})$ and $B(z)=\prod^n_{i=1}\frac{z_{i}-z}{1-\overline{z_{i}}z}$ ($B^{*}(z)=\prod^n_{i=1}\frac{\overline{z_{i}}-z}{1-z_{i}z}$).
By Lemma \ref{Push-forwards}, the fibre $E_{h(S)}(\omega)$ is an $n$-dimensional vector space, more precisely,
\[
E_{h(S_\beta)}(\omega)=\textrm{span}\{\frac{1}{1-\overline{z_i}z};~z_i\in h^{-1}(\omega_0), \ i=1, \cdots, n.\}.
\]
Let
\[
\gamma_i(\omega)=\sum\limits_{k=0}^{\infty}(F(S_\beta))^{-k}M^k_{B}\left(\frac{1}{1-\overline{z_i}z}\right)(\omega-\omega_0)^k \ \ \ \ \text{for} \ i=1, \cdots, n.
\]
Then, each $\gamma_i(\omega)$ is an $H^2_\beta$ valued analytic function on a neighborhood $\Delta$ of $\omega_0$ and $\{\gamma_1(\omega), \cdots, \gamma_n(\omega)\}$ are linear independent for every $\omega\in\Delta$. Since for any $\omega\in\Delta$ and  any $i=1, \cdots, n$, by Lemma \ref{leftin},
\begin{align*}
&(h(S_\beta)-\omega\mathbf{I})(\gamma_i(\omega)) \\
=&((h(S_\beta)-\omega\mathbf{I})-(\omega-\omega_0))(\gamma_i(\omega)) \\
=&(F(S_\beta)B^*(S_\beta)-(\omega-\omega_0))
\left(\sum\limits_{k=0}^{\infty}(F(S_\beta))^{-k}M^k_{B}\left(\frac{1}{1-\overline{z_i}z}\right)(\omega-\omega_0)^k\right) \\
=&0,
\end{align*}
one can see that $\Gamma=\{\gamma_1(\omega), \cdots, \gamma_n(\omega)\}$ is a holomorphic cross-section frame of $E(\Omega)$ over $\Delta$. Therefore, $E_{h(S)}(\Omega)$ is a push-forwards Hermitian holomorphic vector bundle induced by $h$ from the vector bundle $E_S(\mathbb{D})$.
\end{proof}

\begin{remark}
For the above $h$, one can see that $h(S_{\beta})$ is a Cowen-Douglas operator on $\Omega$. More precisely, $h(S_{\beta})\in \mathbf {B}_{n}(\Omega)$.
\end{remark}

\subsection{Main results}

Review that $S_\beta\in\mathbf{B}_{1}(\mathbb{D})$ and for any domain $\Delta$ in $\mathbb{D}$, it also holds that $S_\beta\in\mathbf{B}_{1}(\Delta)$. That is a property of local rigidity, i.e., locality determines the whole. The previously mentioned push-forwards Hermitian holomorphic vector bundles also have local rigidity.

\begin{lemma}[Local rigidity]\label{localrig}
Let $h_1, h_2\in \textrm{Hol}(\overline{\mathbb{D}})$. Suppose that $E_{h_1(S_{\beta})}(\Omega_0)$ and  $E_{h_2(S_{\beta})}(\Omega_0)$ over a domain $\Omega_0$ are similar. Then, $h_1(\mathbb{D})\setminus h_1(\mathbb{T})=h_2(\mathbb{D})\setminus h_2(\mathbb{T})$ and for any domain $\Omega\subseteq h_1(\mathbb{D})\setminus h_1(\mathbb{T})$, $E_{h_1(S_{\beta})}(\Omega)$ and  $E_{h_2(S_{\beta})}(\Omega)$ are similar.
\end{lemma}
\begin{proof}
Suppose that $X$ is a bounded invertible operator such that $XE_{h_1(S_{\beta})}(\Omega_0)=E_{h_2(S_{\beta})}(\Omega_0)$.  Let $\Delta$ be a domain contained in $h_1^{-1}(\Omega_0)$. Since $S_\beta\in\mathbf{B}_{1}(\Delta)$, we have $\textrm{Span}\{E_{S_{\beta}}(\lambda);~\lambda\in\Delta\}=H^2_{\beta}$. Then, by Lemma \ref{Pre-Push-forwards}, we have $\textrm{Span}\{E_{h_i(S_{\beta})}(\omega); ~\omega\in\Omega_0\}=H^2_{\beta}$, $i=1,2$. Furthermore,  $XE_{h_1(S_{\beta})}(\Omega_0)=E_{h_2(S_{\beta})}(\Omega_0)$ implies that for any $\omega\in\Omega_0$ and any $\mathbf{x}\in E_{h_1(S_{\beta})}(\omega)$,
\[
Xh_1(S_{\beta})(\mathbf{x})=X(\omega\mathbf{x})=\omega X(\mathbf{x})=h_2(S_{\beta}) X(\mathbf{x}).
\]
Then, we obtain that $Xh_1(S_{\beta})=h_2(S_{\beta}) X$ and consequently $h_1(\mathbb{D})\setminus h_1(\mathbb{T})=h_2(\mathbb{D})\setminus h_2(\mathbb{T})$. Furthermore, for any $\omega\in h_1(\mathbb{D})\setminus h_1(\mathbb{T})$ and any $\mathbf{y}\in E_{h_1(S_{\beta})}(\omega)$,
\[
h_2(S_{\beta})(X(\mathbf{y}))=Xh_1(S_{\beta})(\mathbf{y})=\omega X(\mathbf{y}).
\]
That means $XE_{h_1(S_{\beta})}(\omega)=E_{h_2(S_{\beta})}(\omega)$.
Therefore, for any domain $\Omega\subseteq h_1(\mathbb{D})\setminus h_1(\mathbb{T})$, the above $X$ is a similarity deformation from $E_{h_1(S_{\beta})}(\Omega)$ to  $E_{h_2(S_{\beta})}(\Omega)$. The proof is finished.
\end{proof}

When consider the similarity classification problem for the Hermitian holomorphic vector bundles induced by analytic functions on $\overline{\mathbb{D}}$, by the above Lemma \ref{localrig}, we could write the vector bundle in brief as $E_{h(S_{\beta})}$ without confusions, or sometimes we could write the vector bundle as $E_{h(S_{\beta})}(\Omega)$ over some certain domain $\Omega$ in $h(\mathbb{D})\setminus h(\mathbb{T})$.

In this paper, our aim is to determine whether the two push-forwards Hermitian holomorphic vector bundles $E_{h_1(S_{\beta})}$ and  $E_{h_2(S_{\beta})}$ are similar. In this study, an analogue of Jordan standard form for push-forwards Hermitian holomorphic vector bundle plays a major role.

\vspace{2mm}

\noindent \textbf{Main Theorem. \ (Jordan decomposition)} \
Suppose that $H^2_{\beta}$ is a weighted Hardy space of polynomial growth. Given any $f\in \textrm{Hol}(\overline{\mathbb{D}})$. There exist a unique positive integer $m$ and an analytic function $h\in \textrm{Hol}(\overline{\mathbb{D}})$ inducing an indecomposable Hermitian holomorphic vector bundle $E_{h(S_{\beta})}$, such that
\[
E_{f(S_\beta)}\sim\bigoplus\limits_1^m E_{h(S_\beta)},
\]
where $h$ is unique in the sense of analytic automorphism group action. That means if there exists another $g\in \textrm{Hol}(\overline{\mathbb{D}})$ inducing an indecomposable vector bundle $E_{g(S_{\beta})}$, such that
\[
E_{f(S_\beta)}\sim\bigoplus\limits_1^m E_{g(S_\beta)},
\]
then there is a M\"{o}bius transformation $\varphi\in \textrm{Aut}(\mathbb{D})$ such that $g=h\circ \varphi$.

\vspace{2mm}

Notice that the indecomposable Hermitian holomorphic vector bundle is an analogue of Jordan block and is unique in the sense of analytic automorphism group action. Furthermore, we give the similarity classification of those push-forwards Hermitian holomorphic vector bundles induced by analytic functions, and give an affirmative answer to Kaplansky's second test problem for those objects.

\vspace{2mm}

\noindent \textbf{Theorem A. \ (Similarity classification)} \
Suppose that $H^2_{\beta}$ is a weighted Hardy space of polynomial growth. Let $h_1, h_2\in \textrm{Hol}(\overline{\mathbb{D}})$. Then, the Hermitian holomorphic vector bundles $E_{h_1(S_{\beta})}$ and  $E_{h_2(S_{\beta})}$ are similar if and only if there are two finite Blaschke products $B_1$ and $B_2$ with the same order and a function $h \in \textrm{Hol}(\overline{\mathbb{D}})$  such that
\[
h_1=h\circ B_1 \ \ \ \text{and} \ \ \ h_2=h\circ B_2,
\]
where the vector bundle $E_{h(S_{\beta})}$ is indecomposable.

\vspace{2mm}

\noindent \textbf{Theorem B. \ (Answer to  Kaplansky's second test problem)} \
Suppose that $H^2_{\beta}$ is a weighted Hardy space of polynomial growth.  Let $h_1, h_2\in \textrm{Hol}(\overline{\mathbb{D}})$. If the direct sums of the Hermitian holomorphic vector bundles $E_{h_1(S_{\beta})}\oplus E_{h_1(S_{\beta})}$ and  $E_{h_2(S_{\beta})}\oplus E_{h_2(S_{\beta})}$ are similar, then $E_{h_1(S_{\beta})}$ and  $E_{h_2(S_{\beta})}$ are similar.

\section{Proof of the main results}

To prove the main results of this paper, we study the cross-section frames and invariant deformations for the push-forwards Hermitian holomorphic vector bundles in subsection \ref{3.1} and subsection \ref{3.2}, respectively. Then, we give the proof of the main theorems in subsection \ref{3.3}.

\subsection{Cross-section frames with Riesz base $\beta$-normalized coefficients}\label{3.1}

Suppose that $E(\Omega)$ is a Hermitian holomorphic vector bundle over a domain $\Omega$. Let $\lambda_0\in\Omega$ and let $\Gamma(\lambda)=\{\gamma_1(\lambda), \gamma_2(\lambda), \ldots, \gamma_n(\lambda)\}$ be a holomorphic cross-section frame of the bundle $E(\Omega)$ on a neighborhood $\Delta\subseteq\Omega$ of $\lambda_0$. For each $i=1,\ldots,n$, write
\[
\gamma_i(\lambda)=\sum\limits_{k=0}^{\infty}\widehat{\gamma_i}(j)(\lambda-\lambda_0)^j.
\]
Notice that each $\widehat{\gamma_i}(j)$ is a vector in the Hilbert space $H^2_{\beta}$. Then, we could write $\widehat{\gamma_i}(j)$ in the following form
\[
\widehat{\gamma_i}(j)=\sum\limits_{k=0}^{\infty}c^{(i,j)}_k z^k,
\]
where $c^{(i,j)}_k\in\mathbb{C}$. Then, we call the sequence of vectors
\[
\mathfrak{F}\triangleq\{\widehat{\gamma_i}(j);~i=1,\ldots,n, ~j=1,2,\ldots\}
\]
the coefficients of the cross-section frame $\Gamma$ over $\lambda_0$, and call the sequence of vectors
\[
\mathfrak{F}_{\beta}\triangleq\{\widehat{\gamma_i}(j)/\beta_j;~i=1,\ldots,n, ~j=1,2,\ldots\}
\]
the $\beta$-normalized coefficients of the cross-section frame $\Gamma$ over $\lambda_0$. Furthermore, if $\mathfrak{F}_{\beta}$ is a Riesz base, we call $\Gamma$ a cross-section frame over $\lambda_0$ with Riesz base $\beta$-normalized coefficients.

In the present paper, the base theory plays an important role. Let us review some basic concepts in base theory on Hilbert space firstly (we refer to \cite{Nik}, more generally, one can see these concepts on Banach space in \cite{Sing}).

\begin{definition}
Let $\mathfrak{X}=\{x_n\}_{n=0}^{\infty}$ be a sequence of vectors in a complex separable Hilbert space $\mathcal{H}$.
\begin{enumerate}
 \item[$(1)$]  The sequence $\mathfrak{X}$ is said to be a Bessel sequence with bound $M$ if for any $x\in\mathcal{H}$,
 \[
 \sum\limits_{n=0}^{\infty}|\langle x, x_n\rangle_{\mathcal{H}}|^2\leq M\|x\|^2.
 \]
 \item[$(2)$]  The sequence $\mathfrak{X}$ is said to be a Riesz-Fischer sequence with bound $m$ if for each sequence $\{c_n\}_{n=0}^{\infty}$ in $\ell^2$ which means the space of square summable sequences,  there corresponds to at least one $x\in\mathcal{H}$ for which
     \[
     \langle x, x_n\rangle_{\mathcal{H}}=c_n \ \ \ \text{and} \ \ \ \|x\|^2\geq m\sum\limits_{n=0}^{\infty}|c_n|^2.
     \]
 \item[$(3)$]  The sequence $\mathfrak{X}$ is said to be a Riesz base sequence if $\mathfrak{X}$ is both a Bessel sequence and a Riesz-Fischer sequence, i.e., there exist two positive numbers $C_1$ and $C_2$ such that for any $x\in\mathcal{H}$,
 \[
 C_1\|x\|^2\leq\sum\limits_{n=0}^{\infty}|\langle x, x_n\rangle_{\mathcal{H}}|^2\leq C_2\|x\|^2.
 \]
 \item[$(4)$]  The sequence $\mathfrak{X}$ is said to be total, if $\mathfrak{X}$ spans the whole space $\mathcal{H}$, i.e., $\textrm{Span}\{\mathfrak{X}\}=\mathcal{H}$.
 \item[$(5)$]  The sequence $\mathfrak{X}$ is said to be a Riesz base, if $\mathfrak{X}$ is a total Riesz base sequence.
 \item[$(5')$] There is also an equivalent definition of Riesz base. The sequence $\mathfrak{X}$ is a Riesz base if there exists a bounded invertible linear operator $V$ such that $\{V(x_n)\}_{n=0}^{\infty}$ is an orthonormal base, i.e., $V(x_m)$ is orthogonal to $V(x_n)$ for all $m\neq n$, and $\|V(x_n)\|=1$ for all $n=0,1,\ldots$.
\end{enumerate}
\end{definition}

Furthermore, following from Bari's work \cite{B51}, one can see an equivalent description of Riesz base by Gram matrix of the sequence.
Let $\mathfrak{X}=\{x_n\}_{n=0}^{\infty}$ and $\mathfrak{Y}=\{y_n\}_{n=0}^{\infty}$ be two sequences of vectors in a complex separable Hilbert space $\mathcal{H}$.
Define
\[
\langle  \mathfrak{X}, \mathfrak{Y}\rangle_{\mathcal{H}}=(\langle  x_j, y_i\rangle_{\mathcal{H}})_{i,j}.
\]
In particular, the Gram matrix of $\mathfrak{X}$ on $\mathcal{H}$ is defined by
\[
\textrm{Gram}_{\mathcal{H}}(\mathfrak{X})=\langle  \mathfrak{X}, \mathfrak{X}\rangle_{\mathcal{H}}=(\langle  x_i, x_j\rangle_{\mathcal{H}})_{i,j}.
\]

\begin{lemma}\label{Grambase}[Theorem 3 in \cite{Sha} and Proposition 9.13 in \cite{Nik}]
Let $\mathfrak{X}=\{x_n\}_{n=0}^{\infty}$ be a sequence of vectors in a complex separable Hilbert space $\mathcal{H}$.
\begin{enumerate}
 \item[$(1)$]  The sequence $\mathfrak{X}$ is a Bessel sequence with bound $M$ if and only if the Gram matrix $\textrm{Gram}_{\mathcal{H}}(\mathfrak{X})$ is a bounded operator on $\ell^2$ with bound $M$.
 \item[$(2)$]  The sequence $\mathfrak{X}$ is a Riesz-Fischer sequence with bound $m$ if and only if the Gram matrix $\textrm{Gram}_{\mathcal{H}}(\mathfrak{X})$ is bounded below on $\ell^2$ with bound $m$.
 \item[$(3)$]  The sequence $\mathfrak{X}$ is a Riesz base sequence if and only if the Gram matrix $\textrm{Gram}_{\mathcal{H}}(\mathfrak{X})$ acting on $\ell^2$ is bounded up and below.
 \item[$(4)$]  The sequence $\mathfrak{X}$ is a Riesz base if and only if $\mathfrak{X}$ is total and the Gram matrix $\textrm{Gram}_{\mathcal{H}}(\mathfrak{X})$ acting on $\ell^2$ is bounded up and below.
\end{enumerate}
\end{lemma}

The sequence $\mathfrak{X}$ is said to be topologically free if none of them is in the closed linear span of the others. A sequence $\mathfrak{X}'=\{x'_n\}_{n=0}^{\infty}$ is called the dual system of a topologically free sequence $\mathfrak{X}$ if
\[
\langle x_i, x'_j\rangle_{\mathcal{H}}=\delta_{ij}\triangleq\left\{
\begin{array}{ll}1, &i=j \\
0, &i\neq j \\
\end{array}\right.
\]
Moreover, if $\mathfrak{X}'\subseteq\textrm{Span}\{\mathfrak{X}\}$, we call $\mathfrak{X}'$ the minimal dual system of $\mathfrak{X}$. Notice that the minimal dual system is unique. In particular, if the sequence $\mathfrak{X}$ spans the whole space, its dual system is unique.
\begin{lemma}\label{dual}[Proposition 9.14 in \cite{Nik}]
Suppose that $\mathfrak{X}=\{x_n\}_{n=0}^{\infty}$ is a topologically free sequence of vectors in a complex separable Hilbert space $\mathcal{H}$ and $\mathfrak{X}'=\{x'_n\}_{n=0}^{\infty}$ is a dual system of $\mathfrak{X}$. If the Gram matrix $\textrm{Gram}_{\mathcal{H}}(\mathfrak{X}')$ is bounded, then the Gram matrix $\textrm{Gram}_{\mathcal{H}}(\mathfrak{X})$ is bounded below. The converse holds if $\mathfrak{X}$ is the minimal dual system of $\mathfrak{X}'$.
\end{lemma}

In particular, if $\mathfrak{X}=\{x_n\}_{n=0}^{\infty}$ is topologically free and spans the whole space $\mathcal{H}$, the Gram matrix $\textrm{Gram}_{\mathcal{H}}(\mathfrak{X}')$ is bounded if and only if the Gram matrix $\textrm{Gram}_{\mathcal{H}}(\mathfrak{X})$ is bounded below.

Let $\mathfrak{E}=\{e_i\}_{i=0}^{\infty}$ be an orthonormal base of $\mathcal{H}$. The sequence $\mathfrak{X}=\{x_j\}_{j=0}^{\infty}$ naturally induces an infinite dimensional matrix $A_{\mathfrak{E},\mathfrak{X}}=(\langle x_j, e_i\rangle_{\mathcal{H}})_{i,j}$, whose $j$-th column represents the vector $x_j$ under the orthonormal base $\mathfrak{E}$. Then, we have $\textrm{Gram}_{\mathcal{H}}(\mathfrak{X})=A^*_{\mathfrak{E},\mathfrak{X}}A_{\mathfrak{E},\mathfrak{X}}$. Moreover, one can see that $A^*_{\mathfrak{E},\mathfrak{X}'}A_{\mathfrak{E},\mathfrak{X}}=\mathbf{I}$ if and only if $\mathfrak{X}'$ is the dual system of $\mathfrak{X}$. Notice that $A_{\mathfrak{E},\mathfrak{X}}$ is just the matrix representation of the operator, which maps $e_j$ to $x_j$ for all $j$, under the orthonormal base $\mathfrak{E}$. For convenience, we also denote the operator by $A_{\mathfrak{E},\mathfrak{X}}$. Obviously, $A_{\mathfrak{E},\mathfrak{X}}$ is a bounded invertible linear operator if and only if $\mathfrak{X}$ is a Riesz base.

In a weighted Hardy space $H^2_{\beta}$, it is obvious that $\mathfrak{E}_{\beta}\triangleq\{z^i/\beta_i\}_{i=0}^{\infty}$ is an orthonormal base of $H^2_{\beta}$. Now we focus on a finite Blaschke product $B(z)=\prod\limits_{j=1}^{m}\frac{z_j-z}{1-\overline{z_j}z}$ which has $m$ distinct zero points $z_1,\cdots,z_m$ and $z_1=0$. Let
\[
\mathfrak{B}=\left\{\frac{B^n(z)}{1-\overline{z_{j}}z}; \ j=1,2,\ldots,m \ \text{and} \ n=0,1,2,\ldots\right\}.
\]
and
\[
\mathfrak{B}_{\beta}=\left\{\frac{1}{1-\overline{z_{j}}z}\frac{B^n(z)}{\beta_n}; \ j=1,2,\ldots,m \ \text{and} \ n=0,1,2,\ldots\right\}.
\]
Denote
\[
D_{\beta}=\begin{bmatrix}
\beta_0 & 0 & 0  & \cdots & 0 & \cdots \\
0   & \beta_1 & 0  & \cdots & 0 & \cdots \\
0   & 0   & \beta_2  & \cdots  & 0 & \cdots \\
\vdots   & \vdots & \vdots & \ddots &\vdots  &\vdots \\
0   & 0 & 0 & \cdots &\beta_k & \cdots \\
\vdots   & \vdots & \vdots&\vdots  &\vdots & \ddots
\end{bmatrix}
\]
and
\[
D_{\beta^{-1}}\otimes\mathbf{I}_m=\begin{bmatrix}
\frac{1}{\beta_0}\mathbf{I}_m & 0 & 0  & \cdots & 0 & \cdots \\
0   & \frac{1}{\beta_1}\mathbf{I}_m & 0  & \cdots & 0 & \cdots \\
0   & 0   & \frac{1}{\beta_2}\mathbf{I}_m  & \cdots  & 0 & \cdots \\
\vdots   & \vdots & \vdots & \ddots &\vdots  &\vdots \\
0   & 0 & 0 & \cdots &\frac{1}{\beta_k}\mathbf{I}_m & \cdots \\
\vdots   & \vdots & \vdots&\vdots  &\vdots & \ddots
\end{bmatrix},
\]
where $\mathbf{I}_m$ is the $m\times m$ identity matrix.
Then
\[
A_{\mathfrak{E}_{\beta},\mathfrak{B}_{\beta}}=D_{\beta}X_{\mathfrak{B}}(D_{\beta^{-1}}\otimes\mathbf{I}_m),
\]
where the $(mn+j)$-th column of $X_{\mathfrak{B}}$ is the Taylor coefficients of $\frac{B^n(z)}{1-\overline{z_{j}}z}$.

\begin{lemma}\label{dualB}
Let $H^2_{\beta}$ be the weighted Hardy space induced by a weight sequence $w=\{w_k\}_{k=1}^{\infty}$ with $w_k\rightarrow 1$. The sequence $\mathfrak{B}_{\beta}$ is total and topologically free in $H^2_{\beta}$. In fact, $\mathfrak{B}_{\beta}$ has a unique dual system.
\end{lemma}

\begin{proof}
For $j=1,2,\ldots,m$ and $n=0,1,2,\ldots$, denote
\[
f_{mn+j}(z)=\frac{1}{1-\overline{z_{j}}z}\frac{B^n(z)}{\beta_n}
\]
and consequently, let
\[
g_{mn+j}(z)=\beta_j\sum\limits_{n=0}^{\infty}\frac{\widehat{f_{mn+j}}(n)}{\beta^2_n}.
\]
and $\mathfrak{G}=\{g_{mn+j}\}_{n=0}^{\infty}$. Notice that
\begin{align*}
\langle \frac{B^n(z)}{1-\overline{z_i}z}, \frac{B^k(z)}{1-\overline{z_j}z}\rangle_{H^2}= \ &\frac{1}{2\pi\mathbf{i}}\int_{\partial\mathbb{D}} \overline{\left(\frac{B^k(z)}{1-\overline{z_j}z}\right)}\cdot \frac{B^n(z)}{1-\overline{z_i}z}\cdot\frac{dz}{z} \\
= \ &\frac{1}{2\pi\mathbf{i}}\cdot\int_{\partial\mathbb{D}}\frac{1}{1-\frac{z_j}{z}}\cdot \frac{B^{n-k}(z)}{1-\overline{z_i}z}\cdot\frac{dz}{z} \\
= \ &\frac{1}{2\pi\mathbf{i}}\cdot\int_{\partial\mathbb{D}}\frac{B^{n-k}(z)}{(z-z_j)(1-\overline{z_i}z)} {dz} \\
= \ &\left\{\begin{array}{cc}
\frac{1}{1-\overline{z_i}z_j}, \ \ \ &\text{if} \ n=k \\
0, \ \ \ &\text{if} \ n\neq k
\end{array}\right..
\end{align*}

Then, under the orthonormal base $\mathfrak{E}_{\beta}\triangleq\{z^i/\beta_i\}_{i=0}^{\infty}$,
\begin{align*}
&\langle \mathfrak{B}_{\beta}, \mathfrak{G} \rangle_{H^2_{\beta}} \\
=&A^*_{\mathfrak{E}_{\beta},\mathfrak{G}}A_{\mathfrak{E}_{\beta},\mathfrak{B}_{\beta}} \\
=&(D_{\beta}(D_{\beta^{-2}}X_{\mathfrak{B}})(D_{\beta}\otimes\mathbf{I}_m))^*D_{\beta}X_{\mathfrak{B}}(D_{\beta^{-1}}\otimes\mathbf{I}_m) \\
=&(D_{\beta}\otimes\mathbf{I}_m)X^*_{\mathfrak{B}}X_{\mathfrak{B}}(D_{\beta^{-1}}\otimes\mathbf{I}_m) \\
=&\begin{bmatrix}
\beta_0\mathbf{I}_m & 0 & 0  & \cdots  \\
0   & \beta_1\mathbf{I}_m & 0  & \cdots  \\
0   & 0   & \beta_2\mathbf{I}_m  & \cdots   \\
\vdots   & \vdots  &\vdots & \ddots
\end{bmatrix}\begin{bmatrix}
A   & 0 & 0 & \cdots \\
0   & A & 0 & \cdots \\
0   & 0   & A  & \cdots \\
\vdots   & \vdots  &\vdots & \ddots
\end{bmatrix}\begin{bmatrix}
\frac{1}{\beta_0}\mathbf{I}_m & 0 & 0  & \cdots  \\
0   & \frac{1}{\beta_1}\mathbf{I}_m & 0  & \cdots  \\
0   & 0   & \frac{1}{\beta_2}\mathbf{I}_m  & \cdots   \\
\vdots   & \vdots  &\vdots & \ddots
\end{bmatrix} \\
=&\begin{bmatrix}
A   & 0 & 0 & \cdots \\
0   & A & 0 & \cdots \\
0   & 0   & A  & \cdots \\
\vdots   & \vdots  &\vdots & \ddots
\end{bmatrix},
\end{align*}
where
\[
A=\begin{bmatrix}
\frac{1}{1-|z_1|^2} & \frac{1}{1-\overline{z_1}z_2} & \frac{1}{1-\overline{z_1}z_3}  & \cdots & \frac{1}{1-\overline{z_1}z_m} \\
\frac{1}{1-\overline{z_2}z_1}   & \frac{1}{1-|z_2|^2} & \frac{1}{1-\overline{z_2}z_3}  & \cdots & \frac{1}{1-\overline{z_2}z_m}  \\
\frac{1}{1-\overline{z_3}z_1}  & \frac{1}{1-\overline{z_3}z_2}   & \frac{1}{1-|z_3|^2}  & \cdots  & \frac{1}{1-\overline{z_3}z_m} \\
\vdots   & \vdots & \vdots & \ddots &\vdots \\
\frac{1}{1-\overline{z_m}z_1}   & \frac{1}{1-\overline{z_m}z_2} & \frac{1}{1-\overline{z_m}z_3} & \cdots &\frac{1}{1-|z_m|^2} \\
\end{bmatrix}.
\]
By Lemma \ref{mmatrix}, $A$ is an invertible matrix, and then we may put
\[
Y=(c_{ik})=A_{\mathfrak{E}_{\beta},\mathfrak{G}}\begin{bmatrix}
A^{-1}   & 0 & 0 & \cdots \\
0   & A^{-1} & 0 & \cdots \\
0   & 0   & A^{-1}  & \cdots \\
\vdots   & \vdots  &\vdots & \ddots
\end{bmatrix}
\]
and let $\mathfrak{Y}=\{y_k(z)\}_{k=1}^{\infty}$,
\[
y_k(z)=\sum\limits_{i=0}^{\infty}c_{ik}\frac{z^i}{\beta_i}.
\]
Then
\[
A^*_{\mathfrak{E}_{\beta},\mathfrak{Y}}A_{\mathfrak{E}_{\beta},\mathfrak{B}_{\beta}}=(A_{\mathfrak{E}_{\beta},\mathfrak{G}}A^{-1})^*A_{\mathfrak{E}_{\beta},\mathfrak{B}_{\beta}}
=(A^{-1})^{*}\langle \mathfrak{B}_{\beta}, \mathfrak{G} \rangle_{H^2_{\beta}}=\mathbf{I},
\]
and consequently $\mathfrak{Y}$ is a dual system of $\mathfrak{B}_{\beta}$ which implies the sequence $\mathfrak{B}_{\beta}$ is topologically free.

Let $\Gamma(\omega)=\{\gamma_j(\omega)\}_{j=1}^m$, where
\[
\gamma_j(\omega)=\sum\limits_{k=0}^{\infty}\left(\frac{B^k(z)}{1-\overline{z_{j}}z} \right)\omega^k.
\]
By Proposition \ref{PFHHVB}, $E_{B(S_{\beta})}(\mathbb{D})$ is a push-forwards Hermitian holomorphic vector bundle induced by $B$ from the vector bundle $E_{S_{\beta}}(\mathbb{D})$ and $\Gamma(\omega)$ is a holomorphic cross-section frame of $E_{B(S_{\beta})}(\mathbb{D})$. Since the closed linear span of $\Gamma(\omega)$ is the whole space $H^2_{\beta}$, so is the closed linear span of $\mathfrak{B}_{\beta}$. Thus, the sequence $\mathfrak{B}_{\beta}$ ($\mathfrak{B}$) is total. Furthermore, the sequence $\mathfrak{Y}$ is the unique dual system of $\mathfrak{B}_{\beta}$ .
\end{proof}

\begin{lemma}\label{CPB}
Let $H^2_{\beta}$ be the weighted Hardy space induced by a weight sequence $w=\{w_k\}_{k=1}^{\infty}$ with $w_k\rightarrow 1$.
Denote $\widetilde{\beta}=\{\widetilde{\beta}_n\}_{n=0}^{\infty}$, where $\widetilde{\beta}_n=(n+1)\beta_n$. Suppose that $B(z)=\prod\limits_{j=1}^{m}\frac{z_j-z}{1-\overline{z_j}z}$ is a finite Blaschke product with $m$ distinct zero points $z_1,\cdots,z_m$ and $z_1=0$. Then,  $\textrm{Gram}_{H^2_{\beta}}(\{\frac{B^{n}}{\beta_n}\}_{n=0}^{\infty})$ is bounded if and only if $\textrm{Gram}_{H^2_{\widetilde{\beta}}}(\{\frac{B^{n}}{\widetilde{\beta}_n}\}_{n=0}^{\infty})$ is bounded.
\end{lemma}

\begin{proof}
Since $B(z)$ is analytic on $\overline{\mathbb{D}}$ and has no zero point on $\partial\mathbb{D}$, the multiplication operator $M_{B}$ is bounded and lower bounded on $H^2_{\beta}$.  Notice that $B'(z)$ is also analytic on $\overline{\mathbb{D}}$ and has no zero point on $\partial\mathbb{D}$ (Theorem 2.1 in \cite{Car}, see also \cite{Fri}). Then the multiplication operator $M_{B'}$ is also bounded and lower bounded on $H^2_{\beta}$.

Define $D_w, D:H^2_{\beta}\rightarrow H^2_{\beta}$ by, for any $f(z)\in H^2_{\beta}$,
\[
D_wf(z)=\sum\limits_{k=0}^{\infty}w_{k+1}\widehat{f}(k)z^k  \ \ \text{and} \ \
Df(z)=\sum\limits_{k=0}^{\infty}\frac{k+2}{k+1}\cdot \widehat{f}(k)z^k .
\]
We could write the two operators $D_w$ and $D$ in matrix form under the orthonormal base $\{z^k/\beta_k\}_{k=0}^{\infty}$ as follows,
\[
D_w=\begin{bmatrix}
w_1 & 0 & 0  & \cdots & 0 & \cdots \\
0   & w_2 & 0  & \cdots & 0 & \cdots \\
0   & 0   & w_3  & \cdots  & 0 & \cdots \\
\vdots   & \vdots & \vdots & \ddots &\vdots  &\vdots \\
0   & 0 & 0 & \cdots &w_k & \cdots \\
\vdots   & \vdots & \vdots&\vdots  &\vdots & \ddots
\end{bmatrix},    \
D=\begin{bmatrix}
2 & 0 & 0  & \cdots & 0 & \cdots \\
0   & \frac{3}{2} & 0  & \cdots & 0 & \cdots \\
0   & 0   & \frac{4}{3}  & \cdots  & 0 & \cdots \\
\vdots   & \vdots & \vdots & \ddots &\vdots  &\vdots \\
0   & 0 & 0 & \cdots &\frac{k+1}{k} & \cdots \\
\vdots   & \vdots & \vdots&\vdots  &\vdots & \ddots
\end{bmatrix}.
\]
One can see that
\begin{align*}
&\langle  \frac{{B}^{i+1}}{\widetilde{\beta_i}}, \frac{{B}^{j+1}}{\widetilde{\beta_{j}}}\rangle_{H^2_{\widetilde{\beta}}} \\
=& \ \frac{1}{\widetilde{\beta_i}\widetilde{\beta_j}}\cdot \sum\limits_{k=0}^{\infty} \overline{\widehat{{B}^{j+1}}(k)}\cdot \widehat{{B}^{i+1}}(k)\cdot \widetilde{\beta_k}^2 \\
=& \ \overline{\widehat{{B}^{j+1}}(0)} \widehat{{B}^{i+1}}(0)+\frac{1}{(i+1)(j+1){\beta}_i{\beta}_j}
\sum\limits_{k=1}^{\infty}(k+1) \overline{\widehat{{B}^{j+1}}(k)}(k+1)\widehat{{B}^{i+1}}(k){\beta}_k^2 \\
=& \ \overline{z_0}^{j+1}z_0^{i+1}+\frac{1}{(i+1)(j+1){\beta}_i{\beta}_j}
\sum\limits_{k=1}^{\infty}(\frac{k+1}{k}\cdot w_k)k\overline{\widehat{{B}^{j+1}}(k)} (\frac{k+1}{k}\cdot w_k)k\widehat{{B}^{i+1}}(k) {\beta}_{k-1}^2 \\
=& \ \frac{1}{(i+1)(j+1){\beta}_i{\beta}_j} \langle  DD_w(B^{i+1})', DD_w(B^{j+1})'\rangle_{H^2_{\beta}} \\
=& \ \frac{1}{{\beta}_i{\beta}_j} \langle  DD_wM_{B'}B^{i}, DD_wM_{B'}B^{j} \rangle_{H^2_{\beta}} \\
=& \ \langle  DD_wM_{B'}(\frac{B^{i}}{\beta_i}), DD_wM_{B'}(\frac{B^{j}}{\beta_j}) \rangle_{H^2_{\beta}}.
\end{align*}
Then
\begin{align*}
\textrm{Gram}_{H^2_{\beta}}\left(\left\{M_{B}(\frac{B^{n}}{\widetilde{\beta}_n})\right\}_{n=0}^{\infty}\right)=&~\textrm{Gram}_{H^2_{\widetilde{\beta}}}
\left(\left\{\frac{B^{n+1}}{\widetilde{\beta}_n}\right\}_{n=0}^{\infty}\right) \\
=&~\textrm{Gram}_{H^2_{\beta}}\left(\left\{DD_wM_{B'}(\frac{B^{n}}{\beta_n})\right\}_{n=0}^{\infty}\right).
\end{align*}
Since all of the operators $D$, $D_w$, $M_{B'}$ and $M_{B}$ are bounded and lower bounded, the Gram matrix $\textrm{Gram}_{H^2_{\widetilde{\beta}}}(\{\frac{B^{n}}{\widetilde{\beta}_n}\}_{n=0}^{\infty})$ is bounded if and only if $\textrm{Gram}_{H^2_{\beta}}(\{\frac{B^{n}}{\beta_n}\}_{n=0}^{\infty})$ is bounded.
\end{proof}

\begin{lemma}[Theorem 7 in \cite{C90}, in essence]\label{Tricontroll}
Let $B(z)=\prod\limits_{j=1}^{m}\frac{z_j-z}{1-\overline{z_j}z}$ be a finite Blaschke product with $m$ distinct zero points $z_1,\cdots,z_m$ and $z_1=0$. Suppose that $H^2_{\alpha}$ and $H^2_{\beta}$ are two weighted Hardy space satisfying
\[
\frac{\alpha_{k+1}}{\alpha_k}\geq\frac{\beta_{k+1}}{\beta_k}, \ \ \ \text{for all} \ k=0,1,2,\ldots.
\]
If the Gram matrix $\textrm{Gram}_{H^2_{\alpha}}(\{\frac{B^{n}}{\alpha_n}\}_{n=0}^{\infty})$ is bounded, then $\textrm{Gram}_{H^2_{\beta}}(\{\frac{B^{n}}{\beta_n}\}_{n=0}^{\infty})$ is also bounded.
\end{lemma}

\noindent \textbf{Key Lemma.} \
Suppose that $H^2_{\beta}$ is a weighted Hardy space of polynomial growth and $B(z)=\prod\limits_{j=1}^{m}\frac{z_j-z}{1-\overline{z_j}z}$ is a finite Blaschke product with $m$ distinct zero points $z_1,\cdots,z_m$ and $z_1=0$. Let
\[
\mathfrak{B}=\left\{\frac{B^n(z)}{1-\overline{z_{j}}z}; \ j=1,2,\ldots,m \ \text{and} \ n=0,1,2,\ldots\right\}
\]
and
\[
\mathfrak{B}_{\beta}=\left\{\frac{1}{1-\overline{z_{j}}z}\frac{B^n(z)}{\beta_n}; \ j=1,2,\ldots,m \ \text{and} \ n=0,1,2,\ldots\right\}.
\]
Then $\mathfrak{B}_{\beta}$ is a Riesz base of $H^2_{\beta}$. Furthermore, $E_{B(S_{\beta})}$ has a cross-section frame $\Gamma(\omega)=\{\gamma_j(\omega)\}_{j=1}^m$ with Riesz base $\beta$-normalized coefficients, where
\[
\gamma_j(\omega)=\sum\limits_{k=0}^{\infty}\left(\frac{B^k(z)}{1-\overline{z_{j}}z} \right)\omega^k.
\]
\begin{proof}
By Lemma \ref{dualB}, the sequence $\mathfrak{B}$ is total and topologically free in $H^2_{\beta}$. Following from Lemma \ref{Grambase}, it suffices to prove the Gram matrix $\textrm{Gram}_{H^2_{\beta}}(\mathfrak{B}_{\beta})$ is both bounded and lower bounded.

\noindent \textbf{Step 1. \ Boundedness of $\textrm{Gram}_{H^2_{\beta}}(\mathfrak{B}_{\beta})$.}

Assume $\sup_k\{(k+1)|w_k-1|\}\leq M\in\mathbb{N}$. Let $H^2_{\widetilde{\beta}}$ be the weighted Hardy space induced by the weight sequence $\widetilde{w}=\{\widetilde{w_k}\}_{k=1}^{\infty}$, where $\widetilde{w_k}=\frac{k+M+1}{k+1}$. Let $H^2_{\widetilde{\beta'}}$ be the weighted Hardy space with $\widetilde{\beta'_k}=(k+1)^M$. Since
\[
\lim\limits_{k\rightarrow\infty}\frac{\widetilde{\beta_k}}{\widetilde{\beta'_k}}=\lim\limits_{k\rightarrow\infty}\frac{\prod_{j=0}^{k}\frac{j+M+1}{j+1}}{(k+1)^M}
=\lim\limits_{k\rightarrow\infty}\frac{\prod_{m=1}^{M}\frac{k+m+1}{m}}{(k+1)^M}=\prod\limits_{m=1}^{M}\frac{1}{m},
\]
$H^2_{\widetilde{\beta}}$ and $H^2_{\widetilde{\beta'}}$ are the same space with two equivalent norms.

As well known, the Gram matrix $\textrm{Gram}_{H^2}(\{{B^{n}}\}_{n=0}^{\infty})$ is bounded for the classical Hardy space $H^2$. Applying Lemma \ref{CPB} $M$ times, we obtain that $\textrm{Gram}_{H^2_{\widetilde{\beta'}}}(\{\frac{B^{n}}{\widetilde{\beta'}_n}\}_{n=0}^{\infty})$ is bounded, and consequently $\textrm{Gram}_{H^2_{\widetilde{\beta}}}(\{\frac{B^{n}}{\widetilde{\beta}_n}\}_{n=0}^{\infty})$ is bounded. Then, it follows from Lemma \ref{Tricontroll} that $\textrm{Gram}_{H^2_{\beta}}(\{\frac{B^{n}}{\beta_n}\}_{n=0}^{\infty})$ is bounded. Let $\mathfrak{X}=\{\frac{B^{n}}{\beta_n}\}_{n=0}^{\infty}$. Then,
\begin{align*}
&\textrm{Gram}_{H^2_{\beta}}(\mathfrak{B}_{\beta}) \\
=& \langle \mathfrak{B}_{\beta}, \mathfrak{B}_{\beta} \rangle_{H^2_{\beta}} \\
=& \begin{bmatrix}
\langle \frac{\mathfrak{X}}{1-\overline{z_{1}}z}, \frac{\mathfrak{X}}{1-\overline{z_{1}}z} \rangle_{H^2_{\beta}} & \langle \frac{\mathfrak{X}}{1-\overline{z_{1}}z}, \frac{\mathfrak{X}}{1-\overline{z_{2}}z} \rangle_{H^2_{\beta}}   & \cdots & \langle \frac{\mathfrak{X}}{1-\overline{z_{1}}z}, \frac{\mathfrak{X}}{1-\overline{z_{m}}z} \rangle_{H^2_{\beta}} \\
\langle \frac{\mathfrak{X}}{1-\overline{z_{2}}z}, \frac{\mathfrak{X}}{1-\overline{z_{1}}z} \rangle_{H^2_{\beta}}   & \langle \frac{\mathfrak{X}}{1-\overline{z_{2}}z}, \frac{\mathfrak{X}}{1-\overline{z_{2}}z} \rangle_{H^2_{\beta}}   & \cdots & \langle \frac{\mathfrak{X}}{1-\overline{z_{2}}z}, \frac{\mathfrak{X}}{1-\overline{z_{m}}z} \rangle_{H^2_{\beta}}  \\
\vdots   & \vdots & \ddots &\vdots \\
\langle \frac{\mathfrak{X}}{1-\overline{z_{m}}z}, \frac{\mathfrak{X}}{1-\overline{z_{1}}z} \rangle_{H^2_{\beta}}   & \langle \frac{\mathfrak{X}}{1-\overline{z_{m}}z}, \frac{\mathfrak{X}}{1-\overline{z_{2}}z} \rangle_{H^2_{\beta}}  & \cdots & \langle \frac{\mathfrak{X}}{1-\overline{z_{m}}z}, \frac{\mathfrak{X}}{1-\overline{z_{m}}z} \rangle_{H^2_{\beta}} \\
\end{bmatrix}.
\end{align*}
Notice that each multiplication operator $M_{\frac{1}{1-\overline{z_{j}}z}}$, $j=1,2,\ldots,m$, is bounded on $H^2_{\beta}$. Thus, the Gram matrix $\textrm{Gram}_{H^2_{\beta}}(\mathfrak{B}_{\beta})$ is bounded.
In addition, we could obtain that the Gram matrix $\textrm{Gram}_{H^2_{\beta^{-1}}}(\mathfrak{B}_{\beta^{-1}})$ is also bounded by a similar method, where
\[
\mathfrak{B}_{\beta^{-1}}=\left\{\frac{\beta_n B^n(z)}{1-\overline{z_{j}}z}; \ j=1,2,\ldots,m \ \text{and} \ n=0,1,2,\ldots\right\}.
\]

\noindent \textbf{Step 2. \ Lower boundedness of $\textrm{Gram}_{H^2_{\beta}}(\mathfrak{B}_{\beta})$.}

By Lemma \ref{dualB}, $\mathfrak{B}_{\beta}$ has a unique dual system $\mathfrak{Y}$. Then, by Lemma \ref{dual}, it suffices to prove the boundedness of $\textrm{Gram}_{H^2_{\beta}}(\mathfrak{Y})$. Review that (in the proof of Lemma \ref{dualB})
\begin{align*}
A_{\mathfrak{E}_{\beta},\mathfrak{Y}}=&D_{\beta}(D_{\beta^{-2}}X_{\mathfrak{B}})(D_{\beta}\otimes\mathbf{I}_m)D_{A^{-1}} \\
=&D_{\beta^{-1}}X_{\mathfrak{B}}(D_{\beta}\otimes\mathbf{I}_m)D_{A^{-1}},
\end{align*}
where
\[
A=\begin{bmatrix}
\frac{1}{1-|z_1|^2} & \frac{1}{1-\overline{z_1}z_2}   & \cdots & \frac{1}{1-\overline{z_1}z_m} \\
\frac{1}{1-\overline{z_2}z_1}   & \frac{1}{1-|z_2|^2}  & \cdots & \frac{1}{1-\overline{z_2}z_m}  \\
\vdots   & \vdots  & \ddots &\vdots \\
\frac{1}{1-\overline{z_m}z_1}   & \frac{1}{1-\overline{z_m}z_2}  & \cdots &\frac{1}{1-|z_m|^2} \\
\end{bmatrix} \ \ \text{and} \ \ D_{A^{-1}}=\begin{bmatrix}
A^{-1}   & 0 & 0 & \cdots \\
0   & A^{-1} & 0 & \cdots \\
0   & 0   & A^{-1}  & \cdots \\
\vdots   & \vdots  &\vdots & \ddots
\end{bmatrix}.
\]
Notice that
\begin{align*}
\textrm{Gram}_{H^2_{\beta}}(\mathfrak{Y})=&A^*_{\mathfrak{E}_{\beta},\mathfrak{Y}}A_{\mathfrak{E}_{\beta},\mathfrak{Y}} \\
=&(D_{\beta^{-1}}X_{\mathfrak{B}}(D_{\beta}\otimes\mathbf{I}_m)D_{A^{-1}})^*D_{\beta^{-1}}X_{\mathfrak{B}}(D_{\beta}\otimes\mathbf{I}_m)D_{A^{-1}} \\
=&D^*_{A^{-1}}\textrm{Gram}_{H^2_{\beta^{-1}}}(\mathfrak{B}_{\beta^{-1}})D_{A^{-1}}.
\end{align*}
Therefore, by the boundedness of $\textrm{Gram}_{H^2_{\beta^{-1}}}(\mathfrak{B}_{\beta^{-1}})$ and $D_{A^{-1}}$, we have that the Gram matrix $\textrm{Gram}_{H^2_{\beta}}(\mathfrak{Y})$ is bounded.

Thus, $\mathfrak{B}_{\beta}$ is a Riesz base of $H^2_{\beta}$. Furthermore, $E_{B(S_{\beta})}$ has a cross-section frame $\Gamma(\omega)=\{\gamma_j(\omega)\}_{j=1}^m$ with Riesz base $\beta$-normalized coefficients, where
\[
\gamma_j(\omega)=\sum\limits_{k=0}^{\infty}\left(\frac{B^n(z)}{1-\overline{z_{j}}z} \right)\omega^k.
\]
\end{proof}

\begin{corollary}\label{autosimilar}
Let $H^2_{\beta}$ be a weighted Hardy space of polynomial growth. For any $\varphi\in\textrm{Aut}(\mathbb{D})$, the Hermitian holomorphic vector bundles $E_{S_{\beta}}$ and  $E_{\varphi(S_{\beta})}$ are similar.
\end{corollary}
\begin{proof}
Without loss of generality, we may assume that $\varphi(z)=\frac{z_0-z}{1-\overline{z_{0}}z}$, where $0\neq z_0\in\mathbb{D}$. Let $B(z)=(z\varphi(z))^*=z\varphi^*(z)$ and
\[
\mathfrak{B}_{\beta}=\left\{\frac{B^n(z)}{\beta_n},~\frac{1}{1-{z_{0}}z}\frac{B^n(z)}{\beta_n}; n=0,1,2,\ldots\right\}.
\]
Then, by the Key Lemma, $\mathfrak{B}_{\beta}$ is a Riesz base of $H^2_{\beta}$. Furthermore, let $X:H^2_{\beta}\rightarrow H^2_{\beta}$ be the unique Riesz base transformation induced by
\[
X\left(\frac{B^n(z)}{\beta_n}\right)=\frac{1}{1-{z_{0}}z}\frac{B^n(z)}{\beta_n} \ \ \text{and} \ \
X\left(\frac{1}{1-{z_{0}}z}\frac{B^n(z)}{\beta_n}\right)=\frac{B^n(z)}{\beta_n}.
\]
Consider the cross-section frame $\gamma(z)=\sum\limits_{k=0}^{\infty}\left(z^k\right)\omega^k$ of $E_{S_{\beta}}$ and the cross-section frame $\gamma_{\varphi}(z)=\sum\limits_{k=0}^{\infty}\left(\frac{(\varphi^*)^k(z)}{1-{z_{0}}z}\right)\omega^k$ of $E_{\varphi(S_{\beta})}$, and the coefficients transformation operator $Y:H^2_{\beta}\rightarrow H^2_{\beta}$  induced by
\[
Y(z^k)=\frac{(\varphi^*)^k(z)}{1-{z_{0}}z} \ \ \ \ \ \ \text{for} \ k=0,1,2,\ldots.
\]
Notice that
\[
Y\left(\frac{B^n(z)}{\beta_n}\right)=\frac{1}{1-{z_{0}}z}\frac{B^n(z)}{\beta_n} \ \ \text{and} \ \
Y\left(\frac{1}{1-{z_{0}}z}\frac{B^n(z)}{\beta_n}\right)=\frac{B^n(z)}{\beta_n}.
\]
Then, by the uniqueness of the Riesz base transformation, $Y$ is just the Riesz base transformation $X$. Therefore, $X$ is the similarity deformation from the Hermitian holomorphic vector bundles $E_{S_{\beta}}$ to $E_{\varphi(S_{\beta})}$.
\end{proof}

\begin{lemma}\label{left}
Let $E_{h_1(S_{\beta})}(\Omega)$ and $E_{h_2(S_{\beta})}(\Omega)$ be two similar Hermitian holomorphic vector bundles over a domain $\Omega$. Then, for any analytic function $h(z)\in\textrm{Hol}(\overline{\Omega})$, the push-forwards vector bundles $E_{(h\circ h_1)(S_{\beta})}$ and  $E_{(h\circ h_2)(S_{\beta})}$ are similar.
\end{lemma}
\begin{proof}
Let $X$ be a similarity deformation from the Hermitian holomorphic vector bundles $E_{h_1(S_{\beta})}(\Omega)$ to $E_{h_2(S_{\beta})}(\Omega)$, i.e., $XE_{h_1(S_{\beta})}(\lambda)=E_{h_2(S_{\beta})}(\lambda)$ for every $\lambda\in\Omega$. Notice that for every $\omega\in h(\Omega)\setminus\left(h(\partial\Omega)\bigcup h(\mathcal{Z}(h'))\right)$ such that $h(z)-\omega$ have $m$ distinct zero points. Then
\[
E_{(h\circ h_1)(S_{\beta})}(\omega)=\textrm{span}\{E_{h_1(S_{\beta})}(z_i);~z_i\in h^{-1}(\omega), \ i=1, \ldots, m.\}
\]
and
\[
E_{(h\circ h_2)(S_{\beta})}(\omega)=\textrm{span}\{E_{h_2(S_{\beta})}(z_i);~z_i\in h^{-1}(\omega), \ i=1, \ldots, m.\}
\]
Hence,
\[
XE_{(h\circ h_1)(S_{\beta})}(\omega)=E_{(h\circ h_2)(S_{\beta})}(\omega),
\]
which implies $X$ is also a similarity deformation from $E_{(h\circ h_1)(S_{\beta})}$ to $E_{(h\circ h_2)(S_{\beta})}$.
\end{proof}

\begin{corollary}\label{Blaschesimilar}
Suppose that $H^2_{\beta}$ is a weighted Hardy space of polynomial growth and $B(z)$ is a finite Blaschke product with order $m$. Then, $E_{B(S_{\beta})}$ is similar to $\bigoplus_{j=1}^{m}E_{S_{\beta}}$.
\end{corollary}
\begin{proof}
Choose $\omega_0\in \mathbb{D}\setminus B(\mathcal{Z}(B'))$. Then, $B(z)-\omega_0$ has $m$ distinct zero points $\lambda_1, \cdots, \lambda_m$. Let
\[
\varphi(z)=\frac{\omega_0-z}{1-\overline{\omega_0}z} \ \ \text{and} \ \ \ \psi(z)=\frac{\lambda_1-z}{1-\overline{\lambda_1}z}.
\]
Let $\widetilde{B}(z)=\varphi\circ B(z)\circ \psi$. Notice that $\widetilde{B}(z)$ has $m$ distinct zero points $z_1,\cdots,z_m$ and $z_1=0$. Then, by the Key Lemma,  $E_{\widetilde{B}(S_{\beta})}$ has a cross-section frame $\Gamma(\omega)=\{\gamma_j(\omega)\}_{j=1}^m$ over $0$ with Riesz base $\beta$-normalized coefficients. Let $\mathfrak{E}=\{e_{j,i}/\beta_i;~j=1,\cdots,m,~i=0,1,\cdots\}$ be the orthonormal base of  $\bigoplus_{j=1}^{m}H^2_{\beta}$, where $e_{j,i}$ is $z^i/\beta_i$ in the $j$-th $H^2_{\beta}$. Furthermore, let $X: \bigoplus_{j=1}^{m}H^2_{\beta}\rightarrow H^2_{\beta}$ be the linear operator defined by mapping $e_{j,i}$ to $\widehat{\gamma_j}(i)/\beta_i$. Then, $X$ is a Riesz base transformation and consequently it is the similarity deformation from $E_{\widetilde{B}(S_{\beta})}$ to $\bigoplus_{j=1}^{m}E_{S_{\beta}}$. Notice that $\varphi\circ\varphi(z)=z$. Then, together with Lemma \ref{left} and Corollary \ref{autosimilar},
\[
E_{B(S_{\beta})}\sim E_{B\circ \psi(S_{\beta})}\sim E_{\varphi\circ\widetilde{B}(S_{\beta})}\sim E_{\varphi(\bigoplus_{j=1}^{m}S_{\beta})}= \bigoplus_{j=1}^{m}E_{\varphi(S_{\beta})}\sim \bigoplus_{j=1}^{m}E_{S_{\beta}}.
\]
This finishes the proof.
\end{proof}

\subsection{Invariant deformations for a Hermitian holomorphic vector bundles}\label{3.2}

Now, let us study the invariant deformations for a Hermitian holomorphic vector bundles.
Firstly, we need the following result, which is a generalization of the results of Thomson \cite{Th77}.

\begin{lemma}[\cite{Th77}, in essence]\label{hFB}
Let $H^2_{\beta}$ be the weighted Hardy space induced by a weight sequence $w=\{w_k\}_{k=1}^{\infty}$ with $w_k\rightarrow 1$.
For any $h\in \textrm{Hol}(\overline{\mathbb{D}})$, there exists a finite Blaschke
product $B$ and a function $F\in \textrm{Hol}(\overline{\mathbb{D}})$ such that
$h=F\circ B$ and $\{M_h \}_{H^2_{\beta}}' = \{M_B \}_{H^2_{\beta}}'$.
\end{lemma}

The above conclusion on the Hardy space was obtained in \cite{Th77}. For more general
results on the Hardy space, see \cite{Th76} and \cite{C78}. The method in \cite{Th77} also works on the weighted
Bergman spaces by replacing the reproducing kernel on the Hardy space by one on the weighted
Bergman spaces $A_{\alpha}^2$ \cite{JZ}. Notice that
\[
k(z,\omega)=\sum\limits_{k=0}^{\infty}\frac{\overline{\omega}^kz^k}{\beta^2_k}
\]
is the reproducing kernel of the weighted Hardy space $H^2_{\beta}$. Similarly, we can give a proof of Lemma \ref{hFB} as the same as the proof
in \cite{Th77} (pp. 524-528), except replacing the reproducing kernel on the classical Hardy space by one on the weighted Hardy space $H^2_{\beta}$.

\begin{lemma}\label{Commutator-S}
Let $H^2_{\beta}$ be the weighted Hardy space induced by a weight sequence $w=\{w_k\}_{k=1}^{\infty}$ with $w_k\rightarrow 1$. Then $\mathcal{L}_{inv}(E_{S_\beta}(\mathbb{D}))\cong H^{\infty}_{\beta^{-1}}$. More precisely, 
\[
\mathcal{L}_{inv}(E_{S_\beta}(\mathbb{D}))=\{g(S_{\beta});~\ g\in H^{\infty}_{\beta^{-1}}\}.
\]
\end{lemma}
\begin{proof}
For any $W\in\mathcal{L}_{inv}(E_{S_\beta})$, we have $WS_{\beta}=S_{\beta}W$ since $S_{\beta}\in\mathbf{B}_1(\mathbb{D})$. Then, $\mathcal{L}_{inv}(E_{S_\beta})=\{S_{\beta}\}'_{H^2_{\beta}}$. Recall that the operator $T_{\beta, \beta^{-1}}$ is the isometry from $H^2_{\beta}$ to $H^2_{\beta^{-1}}$ and its inverse operator $T_{\beta^{-1}, \beta}$ is the isometry from $H^2_{\beta^{-1}}$ to $H^2_{\beta}$, which are defined previously. Consider the operator
\[
T_{\beta, \beta^{-1}}S_{\beta}T_{\beta^{-1}, \beta}: H^2_{\beta^{-1}}\rightarrow H^2_{\beta^{-1}}.
\]
Notice that
\[
T_{\beta, \beta^{-1}}S_{\beta}T_{\beta^{-1}, \beta}=M_z^*.
\]
Then
\[
\{T_{\beta, \beta^{-1}}S_{\beta}T_{\beta^{-1}, \beta}\}'_{H^2_{\beta^{-1}}}=\{M_z^*\}'_{H^2_{\beta^{-1}}}=\{M_g^*; g\in H^{\infty}_{\beta^{-1}}\}.
\]
Consequently,
\begin{align*}
\{S_{\beta}\}'_{H^2_{\beta}}=&\left\{T_{\beta^{-1}, \beta}Y T_{\beta, \beta^{-1}};~Y\in\{T_{\beta, \beta^{-1}}S_{\beta}T_{\beta^{-1}, \beta}\}'_{H^2_{\beta^{-1}}}\right\} \\
=&\left\{T_{\beta^{-1}, \beta}M_g^* T_{\beta, \beta^{-1}};~g\in H^{\infty}_{\beta^{-1}}\right\} \\
=&\left\{g(S_{\beta});~g\in H^{\infty}_{\beta^{-1}}\right\} .
\end{align*}
Therefore, $\mathcal{L}_{inv}(E_{S_\beta}(\mathbb{D}))\cong H^{\infty}_{\beta^{-1}}$.
\end{proof}

\begin{lemma}\label{hFBS}
Let $H^2_{\beta}$ be the weighted Hardy space induced by a weight sequence $w=\{w_k\}_{k=1}^{\infty}$ with $w_k\rightarrow 1$. For any $h\in \textrm{Hol}(\overline{\mathbb{D}})$, there exists a finite Blaschke
product $B$ and a function $F\in \textrm{Hol}(\overline{\mathbb{D}})$ such that
$h=F\circ B$ and $\mathcal{L}_{inv}(E_{h(S_\beta)}(\Omega)) = \mathcal{L}_{inv}(E_{B(S_\beta)}(\mathbb{D}))$. Moreover, $\mathcal{L}_{inv}(E_{h(S_\beta)})
=T_{\beta^{-1}, \beta}\left\{Y^* ;~Y\in\{M_B\}'_{H^2_{\beta^{-1}}} \right\}T_{\beta, \beta^{-1}}$.
\end{lemma}
\begin{proof}
By Lemma \ref{hFB}, there exist a finite Blaschke
product $B$ and a function $F\in \textrm{Hol}(\overline{\mathbb{D}})$ such that
$h=F\circ B$ and $\{M_h \}_{H^2_{\beta^{-1}}}' = \{M_B \}_{H^2_{\beta^{-1}}}'$.  Then, similar to the proof of Lemma \ref{Commutator-S}, one can see that
\begin{align*}
\{h(S_{\beta}) \}'_{H^2_{\beta}}=&\left\{T_{\beta^{-1}, \beta}Y T_{\beta, \beta^{-1}};~Y\in\{T_{\beta, \beta^{-1}}h(S_{\beta})T_{\beta^{-1}, \beta}\}'_{H^2_{\beta^{-1}}}\right\} \\
=&\left\{T_{\beta^{-1}, \beta}Y T_{\beta, \beta^{-1}};~Y\in\{M^*_h\}'_{H^2_{\beta^{-1}}}\right\} \\
=&\left\{T_{\beta^{-1}, \beta}Y T_{\beta, \beta^{-1}};~Y\in\{M^*_B\}'_{H^2_{\beta^{-1}}}\right\} \\
=&\left\{T_{\beta^{-1}, \beta}Y T_{\beta, \beta^{-1}};~Y\in\{T_{\beta, \beta^{-1}}B(S_{\beta})T_{\beta^{-1}, \beta}\}'_{H^2_{\beta^{-1}}}\right\} \\
=&\{B(S_{\beta}) \}_{H^2_{\beta}}' .
\end{align*}
Furthermore,
\[
\mathcal{L}_{inv}(E_{h(S_\beta)})=\{h(S_{\beta}) \}'_{H^2_{\beta}} = \{B(S_{\beta}) \}_{H^2_{\beta}}'= \mathcal{L}_{inv}(E_{B(S_\beta)}).
\]
More precisely,
\begin{align*}
\mathcal{L}_{inv}(E_{h(S_\beta)})=&~\left\{T_{\beta^{-1}, \beta}Y^* T_{\beta, \beta^{-1}} ;~Y\in\{M_B\}'_{H^2_{\beta^{-1}}} \right\} \\
=&~T_{\beta^{-1}, \beta}\left\{Y^* ;~Y\in\{M_B\}'_{H^2_{\beta^{-1}}} \right\}T_{\beta, \beta^{-1}}.
\end{align*}
\end{proof}

\begin{lemma}\label{hSide}
Let $H^2_{\beta}$ be the weighted Hardy space induced by a weight sequence $w=\{w_k\}_{k=1}^{\infty}$ with $w_k\rightarrow 1$. Let $h\in \textrm{Hol}(\overline{\mathbb{D}})$. Then $\mathcal{L}_{inv}(E_{h(S_\beta)})$ has no nontrivial idempotent if and only if $\mathcal{L}_{inv}(E_{h(S_\beta)}) = \mathcal{L}_{inv}(E_{S_\beta})$.
\end{lemma}
\begin{proof}
For any finite Blaschke product $B(z)$, it is well-known that $\{M_B\}'_{H^2_{\beta^{-1}}}$ has no nontrivial idempotent if and only if $B(z)$ is just a M\"{o}bius transformation. Then, by Lemma \ref{Commutator-S} and Lemma \ref{hFBS}, $E_{h(S_\beta)}$ is indecomposable if and only if $\mathcal{L}_{inv}(E_{h(S_\beta)})\cong H^{\infty}_{\beta^{-1}}$. This finishes the proof.
\end{proof}

\begin{lemma}\label{Simforindecomp}
Let $H^2_{\beta}$ be the weighted Hardy space induced by a weight sequence $w=\{w_k\}_{k=1}^{\infty}$ with $w_k\rightarrow 1$. Let $h_1, h_2\in \textrm{Hol}(\overline{\mathbb{D}})$. Suppose that $E_{h_1(S_\beta)}$ and $E_{h_2(S_\beta)}$ are indecomposable. Then, $E_{h_1(S_\beta)}$ and $E_{h_2(S_\beta)}$ are similar if and only if there exists a M\"{o}bius transformation $\varphi$ such that $h_2=h_1\circ \varphi$.
\end{lemma}
\begin{proof}
Suppose that there is a M\"{o}bius transformation $\varphi$ such that $h_2=h_1\circ \varphi$. It follows from Corollary \ref{autosimilar} that
$E_{\varphi(S_\beta)}$ is similar to $E_{S_\beta}$.
Then, by Corollary \ref{left},
\[
E_{h_2(S_\beta)}=E_{h_1\circ \varphi(S_\beta)}\sim E_{h_1\circ {\textrm Id}(S_\beta)}=E_{h_1(S_\beta)}.
\]

Now, suppose that $\widetilde{X}$ is a bounded invertible linear operator on $H^2_{\beta}$ such that
\[
\widetilde{X}(E_{h_1(S_\beta)})= E_{h_2(S_\beta)}.
\]
Then,
\[
\widetilde{X} \mathcal{L}_{inv}(E_{h(S_\beta)})  \widetilde{X}^{-1}= \mathcal{L}_{inv}(E_{h(S_\beta)}).
\]
Let
\[
X=T_{\beta, \beta^{-1}}X^* T_{\beta^{-1}, \beta}:H^2_{\beta^{-1}}\rightarrow H^2_{\beta^{-1}}.
\]
Then, by Lemma \ref{hFBS},
\[
X \{M_{h_1} \}_{H^2_{\beta^{-1}}}'  X^{-1}= \{M_{h_2} \}_{H^2_{\beta^{-1}}}'.
\]

Since $\{M_{h_1} \}_{H^2_{\beta^{-1}}}' = \{M_{h_2} \}_{H^2_{\beta^{-1}}}' = \{M_z \}_{H^2_{\beta^{-1}}}'=H^{\infty}_{\beta^{-1}}$,
there exists a function $g\in H^{\infty}_{\beta^{-1}}$ such that
\[
X M_{z} X^{-1}= M_{g}.
\]
Then, there exists a function $g:\mathbb{D}\rightarrow\mathbb{D}$ in $H^{\infty}_{\beta^{-1}}$ and a function $u\neq \mathbf{0}$ in $H^{2}_{\beta^{-1}}$ such that $X(f)=u\cdot(f\circ g)$ for any $f\in H^2_{\beta^{-1}}$. Similarly, there exists a function $\psi:\mathbb{D}\rightarrow\mathbb{D}$ in $H^{\infty}_{\beta^{-1}}$ and a function $v\neq\mathbf{0}$ in $H^{2}_{\beta^{-1}}$ such that $X^{-1}(f)=v\cdot(f\circ \psi)$ for any $f\in H^2_{\beta^{-1}}$. Notice that for any $f\in H^2_{\beta^{-1}}$,
\[
f=XX^{-1}(f)=X(v\cdot(f\circ \psi))=u\cdot v(g)\cdot f(\psi\circ g)
\]
and
\[
f=X^{-1}X(f)=X^{-1}(u\cdot(f\circ g))=v\cdot u(\psi)\cdot f(g\circ \psi).
\]
If we choose $f(z)\equiv1$, then
\[
u\cdot v(g)\equiv 1 \ \ \ \text{and} \ \ \ v\cdot u(\psi)\equiv 1.
\]
Consequently, we have
\[
g\circ\psi=\psi\circ g={\textrm Id}_{\mathbb{D}}.
\]
Then $g$ is an analytic automorphism on $\mathbb{D}$. Furthermore, following from
\[
u\cdot(h_1\circ g)\cdot(f\circ g)=XM_{h_1}(f)= M_{h_2}X(f)= h_2\cdot u\cdot(f\circ g) \ \ \ \text{for any} \ f\in H^2_{\beta^{-1}},
\]
we obtain that
\[
h_2=h_1\circ g.
\]
The proof is finished.
\end{proof}

\begin{lemma}\label{Bundledec}
Let $H^2_{\beta}$ be a weighted Hardy space of polynomial growth. Given any $f\in \textrm{Hol}(\overline{\mathbb{D}})$. Suppose that $f=h\circ B$, where $h\in \textrm{Hol}(\overline{\mathbb{D}})$ and $B$ is a finite Blaschke
product with order $m$ such that $\mathcal{L}_{inv}(E_{f(S_\beta)}) = \mathcal{L}_{inv}(E_{B(S_\beta)})$. Then
\[
E_{f(S_\beta)}\sim\bigoplus\limits_1^m E_{h(S_\beta)},
\]
$\mathcal{L}_{inv}(E_{h(S_\beta)})= \mathcal{L}_{inv}(E_{S_\beta})$ and $\mathcal{L}_{inv}(E_{h(S_\beta)})$ is indecomposable.
\end{lemma}
\begin{proof}
By Lemma \ref{Blaschesimilar}, $E_{B(S_{\beta})}$ is similar to $\bigoplus_{j=1}^{m}E_{S_{\beta}}$. Then
\[
E_{f(S_\beta)}\sim E_{h(\bigoplus_{j=1}^{m}S_\beta)}=\bigoplus\limits_1^m E_{h(S_\beta)}.
\]
By Lemma \ref{hFBS} and Lemma \ref{Commutator-S}, one can see that
\[
\mathcal{L}_{inv}(E_{h(S_\beta)})= \mathcal{L}_{inv}(E_{S_\beta})\cong H^{\infty}_{\beta^{-1}}.
\]
Furthermore, by Lemma \ref{hSide}, we have $\mathcal{L}_{inv}(E_{h(S_\beta)})$ is indecomposable.
\end{proof}

\subsection{Proof of the main theorems}\label{3.3}

In the study the similarity of bounded linear operators, K-theory of Banach algebra is a powerful technique, for instance, a similarity classification of Cowen-Douglas operators was given by using the ordered K-group of the commutant algebra as an invariant \cite{J04} and \cite{JGJ}. The technique of K-theory is also useful to prove our main theorems. Let us give a quick review for the K-theory of Banach algebra.

Let $\mathcal{A}$ be a Banach algebra and $Proj(\mathcal{A})$ be the set of all idempotents in $\mathcal{A}$. The algebraic equivalence "$\sim_{a}$" is introduced in $Proj(\mathcal{A})$. Let $e$ and $\widetilde{e}$ be two elements in $Proj(\mathcal{A})$. We say that $e\sim_{a} \widetilde{e}$ if there are two elements $x,y \in\mathcal{A}$ such that
\[
xy=e \ \ \ \text{and} \ \ \ yx=\widetilde{e}.
\]
Let ${\textbf Proj}(\mathcal{A})$ denote the algebraic equivalence classes of $Proj(\mathcal{A})$ under algebraic equivalence "$\sim_{a}$". Let
\[
{\textrm M}_{\infty}(\mathcal{A}) = \bigcup\limits_{n=1}^{\infty} {\textrm M}_{n}(\mathcal{A}),
\]
where ${\textrm M}_{n}(\mathcal{A})$ is the algebra of $n\times n$ matrices with entries in $\mathcal{A}$. Set
\[
\bigvee(\mathcal{A})={\textbf Proj}({\textrm M}_{\infty}(\mathcal{A})).
\]
Then $\bigvee({\textrm M}_{n}(\mathcal{A}))$ is isomorphic to $\bigvee(\mathcal{A})$.
The direct sum of two matrices gives a natural addition in ${\textrm M}_{\infty}(\mathcal{A})$ and hence induces an addition "$+$" in ${\textbf Proj}({\textrm M}_{\infty}(\mathcal{A}))$ by
\[
[p] + [q] = [p  \oplus q],
\]
where $[p]$ denotes the equivalence class of the idempotent $p$. Furthermore, $(\bigvee(\mathcal{A}), +)$ forms a semigroup and depends on $\mathcal{A}$
only up to stable isomorphism, and then $K_0(\mathcal{A})$ is the Grothendieck group of $\bigvee(\mathcal{A})$.

\begin{theorem}[\cite{CFJ}, see also in \cite{JW}]\label{CFJ-K}
Let $T$ be a bounded operator on a Hilbert space $\mathcal{H}$. The following are equivalent:
\begin{enumerate}
\item [(1)] $T$ is similar to $\sum_{i=1}^k \oplus A_i^{(n_i)}$ under the space decomposition $\mathcal{H}=\sum_{i=1}^k \oplus \mathcal{H}_i^{(n_i)}$,
where $k$ and $n_i$ are finite, $A_i$ is strongly irreducible and $A_i$ is not similar to $A_j$ if $i\neq j$.
Moreover, $T^{(n)}$ has a unique strongly irreducible decomposition up to similarity.
\item [(2)] The semigroup $\bigvee (\{T\}')$ is isomorphic to the semigroup $\mathbb{N}^{(k)}$, where $\mathbb{N}$ is the set of all natural numbers $\{0, 1, 2,\ldots\}$ and the isomorphism $\phi$ sends
\[
[I]\rightarrow n_1e_1 + n_2e_2 +\cdots+n_ke_k,
\]
where $\{e_i\}_{i=1}^k$ are the generators of $\mathbb{N}^{(k)}$ and $n_i\neq 0$.
\end{enumerate}
\end{theorem}

\begin{corollary}[\cite{CFJ}, see also in \cite{JW}]\label{n1}
Let $T_1$ and $T_2$ be two strongly irreducible operators on a Hilbert space $\mathcal{H}$, and $T\sim T_1^{(n_1)}\oplus T_2^{(n_2)}$.
If $\bigvee (\{T\}')$ is isomorphic to the semigroup $\mathbb{N}$, then $T_1$ is similar to $T_2$.
\end{corollary}

Following from the preliminaries in the previous subsections, we could prove Main Theorem now.

\begin{proof}[\textbf{Proof of Main Theorem}]
By Lemma \ref{hFB} and Lemma \ref{Bundledec}, there exists a finite Blaschke
product $B$ with order $m$ and an analytic function $h\in \textrm{Hol}(\overline{\mathbb{D}})$ inducing an indecomposable Hermitian holomorphic vector bundle $E_{h(S_{\beta})}$ such that
\[
E_{f(S_\beta)}\sim\bigoplus\limits_1^m E_{h(S_\beta)}.
\]

Now, suppose that there are two finite Blaschke products $B_i$ with order $m_i$ and analytic functions $h_i \in \textrm{Hol}(\overline{\mathbb{D}})$ inducing an indecomposable Hermitian holomorphic vector bundle $E_{h_i(S_{\beta})}$, $i=1,2$,  such that
\[
E_{f(S_\beta)}\sim\bigoplus\limits_1^{m_i} E_{h_i(S_\beta)}.
\]
Notice that $\bigoplus\limits_1^{m_1} E_{h_1(S_\beta)}$ and $\bigoplus\limits_1^{m_2} E_{h_2(S_\beta)}$ are the Hermitian holomorphic vector bundles induced by the Cowen-Douglas operators $\bigoplus\limits_1^{m_1} h_1(S_\beta)$ and $\bigoplus\limits_1^{m_2} h_2(S_\beta)$ on $\Omega$, respectively.
Then, by
\[
\bigoplus\limits_1^{m_1} E_{h_1(S_\beta)}(\Omega)\sim\bigoplus\limits_1^{m_2} E_{h_2(S_\beta)},
\]
we have
\[
\bigoplus\limits_1^{m_1} h_1(S_\beta)\sim\bigoplus\limits_1^{m_2} h_2(S_\beta).
\]
Let $T=h_1(S_\beta)^{(m_1)}\oplus h_2(S_\beta)^{(m_2)}$. Then,
\[
T\sim h_1(S_\beta)^{(2m_1)}.
\]
Since
\[
\{h_1(S_\beta)^{(2m_1)} \}_{H^2_{\beta}}' =\{M_F; F\in \textrm {M}_{2m_1}(H^{\infty}_{\beta})\},
\]
it follows from Lemma 2.9 in \cite{CFJ} or Theorem 6.11 in \cite{JW} that
\[
\bigvee(\{T \}_{H^2_{\beta}}')\cong \bigvee(\{M_{h_1}^{(2m_1)} \}_{H^2_{\beta}}') \cong \bigvee({\textrm M_{2m_1}}(H^{\infty}_{\beta}))\cong \bigvee (H^{\infty}_{\beta}) \cong\mathbb{N}.
\]
In addition, since $E_{h_1(S_\beta)}(\Omega)$ is indecomposable and
\[
\{h_1(S_\beta)\}'_{H^2_{\beta}}=\mathcal{L}_{inv}(E_{h_1(S_\beta)}),
\]
$h_1(S_\beta)$ is a strongly irreducible operator and so is $h_2(S_\beta)$.
By Corollary \ref{n1}, we have $h_1(S_\beta)\sim h_2(S_\beta)$ and consequently $E_{h_1(S_\beta)}\sim E_{h_2(S_\beta)}$. Thus, by Lemma \ref{Simforindecomp},  there exists a M\"{o}bius transformation $\varphi$ such that $h_2=h_1\circ \varphi$.

Furthermore, notice that for any $\omega\in\Omega$
\[
m_1\cdot\textrm{dim}E_{F_1(S_\beta)}(\omega)=\textrm{dim}E_{h_1(S_\beta)}(\omega)=\textrm{dim}E_{h_2(S_\beta)}(\omega)=m_2\cdot\textrm{dim}E_{F_2(S_\beta)}(\omega).
\]
Since $\textrm{dim}E_{F_1(S_\beta)}(\omega)=\textrm{dim}E_{F_2(S_\beta)}(\omega)$,
we have $m_1=m_2$.
\end{proof}

Furthermore, we could prove Theorem A and Theorem B.

\begin{proof}[\textbf{Proof of Theorem A}]

Let us prove the sufficiency first.  Suppose that there are two finite Blaschke products $B_1$ and $B_2$ with the same order $m$ and a function $h \in \textrm{Hol}(\overline{\mathbb{D}})$  such that
\[
h_1=h\circ B_1 \ \ \ \text{and} \ \ \ h_2=h\circ B_2,
\]
where $E_{h_1(S_\beta)}$ and $E_{h_2(S_\beta)}$ are indecomposable Hermitian holomorphic vector bundles. By Main Theorem (or Lemma \ref{Bundledec}), we have
\[
E_{h_1(S_\beta)}\sim\bigoplus\limits_1^m E_{h(S_\beta)}\sim E_{h_2(S_\beta)}.
\]

Now we will show the necessity. Suppose that the bundle $E_{h_1(S_\beta)}$ is  similar to the  bundle $E_{h_2(S_\beta)}$.
For $i=1,2$, write $h_i=g_i\circ B_i$, where $g_i\in \textrm{Hol}(\overline{\mathbb{D}})$ and $B_i$ is a finite Blaschke
product with order $m_i$ such that $\mathcal{L}_{inv}(E_{f_i(S_\beta)}) = \mathcal{L}_{inv}(E_{B_i(S_\beta)})$. By Main Theorem (or Lemma \ref{Bundledec}), we have
\[
E_{h_1(S_\beta)}\sim\bigoplus\limits_1^{m_1} E_{g_1(S_\beta)},\ \ \ \text{and} \ \ \ E_{h_2(S_\beta)}\sim\bigoplus\limits_1^{m_2} E_{g_2(S_\beta)},
\]
where both $E_{g_1(S_\beta)}$ and $E_{g_2(S_\beta)}$ are indecomposable. Following from the uniqueness in Main Theorem, we have $m_1=m_2$ and there is a M\"{o}bius transformation $\varphi\in \textrm{Aut}(\mathbb{D})$ such that $g_2=g_1\circ \varphi$. Consequently,
\[
h_2=g_2\circ B_2=(g_1\circ \varphi)\circ B_2=g_1\circ (\varphi\circ B_2).
\]
Notice that $\varphi\circ B_2$ is also a finite Blaschke product with the same order as $B_2$. This finishes the proof.
\end{proof}

\begin{proof}[\textbf{Proof of Theorem B}]
For $i=1,2$, it follows from the Main Theorem that,
there exists a unique positive integer $m_i$ and an analytic function $g_i\in \textrm{Hol}(\overline{\mathbb{D}})$ inducing an indecomposable vector bundle $E_{g_i(S_{\beta})}$, such that
\[
E_{h_i(S_\beta)}\sim\bigoplus\limits_1^{m_i} E_{g_i(S_\beta)}.
\]
Since the direct sums of the Hermitian holomorphic vector bundles $E_{h_1(S_{\beta})}\oplus E_{h_1(S_{\beta})}$ and  $E_{h_2(S_{\beta})}\oplus E_{h_2(S_{\beta})}$ are similar, we have
\[
\bigoplus\limits_1^{2m_1} E_{g_1(S_\beta)}\sim E_{h_1(S_{\beta})}\oplus E_{h_1(S_{\beta})}\sim E_{h_2(S_{\beta})}\oplus E_{h_2(S_{\beta})}\sim\bigoplus\limits_1^{2m_2} E_{g_2(S_\beta)}.
\]
Following from the uniqueness in Main Theorem, we have $2m_1=2m_2$ and there is a M\"{o}bius transformation $\varphi\in \textrm{Aut}(\mathbb{D})$ such that $g_2=g_1\circ \varphi$. Therefore,
\[
E_{h_1(S_{\beta})}\sim\bigoplus\limits_1^{m_1} E_{g_1(S_\beta)}\sim\bigoplus\limits_1^{m_2} E_{g_2(S_\beta)}\sim E_{h_2(S_{\beta})}.
\]
\end{proof}

\section{Some remarks and applications}

In this section, we give some remarks and applications. One hand, we will discuss the case of intermediate growth to show the necessity of the setting of polynomial growth condition. On the other hand, as applications, we answer the problem proposed by R. Douglas in 2007, and obtain the $K_0$-group of the commutant algebra of a multiplication operator on a weighted Hardy space of polynomial growth.

\begin{lemma}\label{cs-cs}
Let $H^2_{\beta}$ be the weighted Hardy space induced by a weight sequence $w=\{w_k\}_{k=1}^{\infty}$ with $w_k\rightarrow 1$. Let $\varphi\in \textrm{Aut}(\mathbb{D})$
Suppose that $X\in\mathcal{L}(H^2_{\beta})$ is a similarity deformation from the vector bundle $E_{S_\beta}(\mathbb{D})$ to the vector bundle $E_{\varphi(S_\beta)}(\mathbb{D})$. Then, the multiplication operators $M_z$ and $M_{\varphi}$ are similar on $H^2_{\beta^{-1}}$.
\end{lemma}
\begin{proof}
Since $S_{\beta}\in\mathbf{B}_1(\mathbb{D})$, we have $XS_{\beta}=\varphi(S_\beta)X$ and consequently
\[
(T_{\beta, \beta^{-1}}XT_{\beta^{-1},\beta})(T_{\beta, \beta^{-1}}S_{\beta}T_{\beta^{-1},\beta})=(T_{\beta, \beta^{-1}}\varphi(S_\beta)T_{\beta^{-1},\beta})(T_{\beta,\beta^{-1}}X T_{\beta^{-1},\beta}).
\]
Following from Lemma \ref{Commutator-S}, one can see that
\[
(T_{\beta, \beta^{-1}}S_{\beta}T_{\beta^{-1},\beta})^*=M_z \ \ \ \text{and} \ \ \ (T_{\beta, \beta^{-1}}\varphi(S_\beta)T_{\beta^{-1},\beta})^*=M_{\varphi}.
\]
Then, for any $n=0,1,\cdots$,
\[
M_z^n(T_{\beta, \beta^{-1}}XT_{\beta^{-1},\beta})^*=(T_{\beta,\beta^{-1}}X T_{\beta^{-1},\beta})^*M_{\varphi}^n.
\]
Since $(T_{\beta, \beta^{-1}}XT_{\beta^{-1},\beta})^*$ is a bounded invertible linear operator, the multiplication operators $M_z$ and $M_{\varphi}$ are similar on $H^2_{\beta^{-1}}$.
\end{proof}

\begin{lemma}\label{CorA}
Suppose that $H^2_{\beta}$ is the weighted Hardy space induced by a weight sequence $w=\{w_k\}_{k=1}^{\infty}$ with $w_k\rightarrow 1$. Let $\varphi(z)=\frac{z_0-z}{1-\overline{z_0}z}$, $z_0\in\mathbb{D}$, be a M\"{o}bius transformation on $\mathbb{D}$. Denote $\mathfrak{F}=\{\frac{\varphi^n(z)}{1-\overline{z_0}z}\}_{n=0}^{\infty}$, $\mathfrak{F}_{\beta}=\{\frac{1}{1-\overline{z_0}z}\frac{\varphi^n(z)}{\beta_n}\}_{n=0}^{\infty}$ and $\mathfrak{F}_{\beta^{-1}}=\{\frac{\beta_n\varphi^n(z)}{1-\overline{z_0}z}\}_{n=0}^{\infty}$.
Then, $\mathfrak{F}_{\beta}$ is a  Riesz base of $H^2_{\beta}$
if and only if $\mathfrak{F}_{\beta^{-1}}$ is a Riesz base of $H^2_{\beta^{-1}}$.
\end{lemma}

\begin{proof}
Assume that $\mathfrak{F}_{\beta}=\{\frac{1}{1-\overline{z_0}z}\frac{\varphi^n(z)}{\beta_n}\}_{n=0}^{\infty}$ is a Riesz base of $H^2_{\beta}$. Then, $\mathfrak{F}_{\beta}$ is a total sequence in $H^2_{\beta}$ and the infinite dimensional matrix $A_{\mathfrak{E}_{\beta},\mathfrak{F}_{\beta}}=D_{\beta}X_{\mathfrak{F}}D^{-1}_{\beta}$ is a bounded invertible linear operator on $\ell^2$, where the $n$-th column of $X_{\mathfrak{F}}$ is the Taylor coefficients of $\frac{\varphi^n(z)}{1-\overline{z_{0}}z}$. In addition, $A_{\mathfrak{E}_{\beta},\mathfrak{F}_{\beta^{-1}}}=D^{-1}_{\beta}X_{\mathfrak{F}}D_{\beta}$, where $\mathfrak{F}_{\beta^{-1}}=\{\frac{\beta_n\varphi^n(z)}{1-\overline{z_0}z}\}_{n=0}^{\infty}$. Notice that $\{\frac{\sqrt{1-|z_0|^2}}{1-\overline{z_0}z}\varphi^n(z)\}_{n=0}^{\infty}$ is an orthonormal base in the classical Hardy space $H^2$, i.e., $\textrm{Gram}_{H^2}(\mathfrak{F})=(1-|z_0|^2)\mathbf{I}$. Then,
\begin{align*}
A^*_{\mathfrak{E}_{\beta},\mathfrak{F}_{\beta^{-1}}}A_{\mathfrak{E}_{\beta},\mathfrak{F}_{\beta}}
=&(D_{\beta}X^*_{\mathfrak{F}}D^{-1}_{\beta})(D_{\beta}X_{\mathfrak{F}}D^{-1}_{\beta}) \\
=& D_{\beta}X^*_{\mathfrak{F}}X_{\mathfrak{F}}D^{-1}_{\beta} \\
=& D_{\beta}(1-|z_0|^2)\mathbf{I}D^{-1}_{\beta} \\
=& (1-|z_0|^2)\mathbf{I}.
\end{align*}
Consequently, we have that $A^*_{\mathfrak{E}_{\beta},\mathfrak{F}_{\beta^{-1}}}$ is bounded invertible and so is
$A_{\mathfrak{E}_{\beta},\mathfrak{F}_{\beta^{-1}}}$.
Applying Lemma \ref{Grambase} again, we obtain that $\mathfrak{F}_{\beta^{-1}}$ is a Riesz base of $H^2_{\beta^{-1}}$. The converse could be proven in a similar way.
\end{proof}

Recall that if the multiplication operator $M_z:H^2_{\beta}\rightarrow H^2_{\beta}$ is strictly cyclic, then $H^{\infty}_{\beta}=H^2_{\beta}$.

\begin{lemma}\label{similaroperator}
Let $H^2_{\beta}$ be the weighted Hardy space induced by a weight sequence $w=\{w_k\}_{k=1}^{\infty}$ with $w_k\rightarrow 1$, and let $\varphi(z)=\frac{z_0-z}{1-\overline{z_0}z}$, $z_0\in\mathbb{D}$, be a M\"{o}bius transformation on $\mathbb{D}$. Denote $\mathfrak{F}=\{\frac{\varphi^n(z)}{1-\overline{z_0}z}\}_{n=0}^{\infty}$ and $\mathfrak{F}_{\beta}=\{\frac{1}{1-\overline{z_0}z}\frac{\varphi^n(z)}{\beta_n}\}_{n=0}^{\infty}$. Suppose that the multiplication operator $M_z:H^2_{\beta}\rightarrow H^2_{\beta}$ is strictly cyclic. If $M_z$ is similar to $M_{\varphi}$, then
$\mathfrak{F}_{\beta}$ is a Riesz base of $H^2_{\beta}$.
\end{lemma}

\begin{proof}
Suppose that $T:H^2_{\beta}\rightarrow H^2_{\beta}$ is a bounded invertible linear operator such that $TM_z=M_{\varphi}T$. Write $f(z)=T(\mathbf{1})$. Then for any $n=0,1,\cdots$,
\[
T(z^n)=f(z)\varphi^n(z).
\]
Moreover, according to $\varphi\circ\varphi(z)=z$, one can see that
\[
T^2(z^n)=T(f(z)\varphi^n(z))=f(z)f(\varphi(z))z^n,
\]
and consequently $T^2$ is just the multiplication operator $M_{f(z)f(\varphi(z))}$.
Since $T^2$ is a bounded invertible operator, the function $f(z)f(\varphi(z))$ is an invertible element in $H^{\infty}_{\beta}=H^2_{\beta}$. Then, by $f(z)\in H^2_{\beta}=H^{\infty}_{\beta}$, we have
\[
\frac{1}{f^2(\varphi(z))}=\frac{f^2(z)}{\left(f(z)f(\varphi(z))\right)^2}\in H^2_{\beta}.
\]
Furthermore,
\[
\frac{1}{f(z)}=T\left(\frac{1}{f^2(\varphi(z))}\right)\in  H^2_{\beta},
\]
that implies $f(z)$ is also an invertible element in $H^{\infty}_{\beta}$. The bounded invertible linear operator $T$ maps the orthonormal base $\{\frac{z^n}{\beta_n}\}_{n=0}^{\infty}$ to the Riesz base $\{\frac{f(z)\varphi^n(z)}{\beta_n}\}_{n=0}^{\infty}$, and then the bounded invertible linear operator $M_{\frac{1}{(1-\overline{z_0}z)f(z)}}$ maps the Riesz base $\{\frac{f(z)\varphi^n(z)}{\beta_n}\}_{n=0}^{\infty}$ to the Riesz base $\{\frac{1}{1-\overline{z_0}z}\frac{\varphi^n(z)}{\beta_n}\}_{n=0}^{\infty}$. Therefore, $\mathfrak{F}_{\beta}$ is a Riesz base of $H^2_{\beta}$.
\end{proof}

Now, let us review the weighted Bergman space. Let $dA$ denote the Lebesgue area measure on the unit open disk $\mathbb{D}$, normalized so that the measure of $\mathbb{D}$ equals $1$. For $\alpha\geq0$, the weighted Bergman space $A_{\alpha}^2$ is the space of analytic functions on $\mathbb{D}$ which are square-integrable with respect to the measure $dA_{\alpha}(z)=(\alpha+1)(1-|z|^2)^{\alpha}dA(z)$. As well known, it could be seemed as a weighted Hardy space $H^2_{\beta^{(\alpha)}}$, where $\beta^{(\alpha)}_0=1$ and for $k=1,2,\ldots$, $\beta_k^{(\alpha)}=\prod\limits_{j=1}^{k}w_{j}^{(\alpha)}$ with $w_{j}^{(\alpha)}=\sqrt{\frac{j+1}{j+2\alpha+1}}$. Since
$$
\lim\limits_{j\rightarrow\infty}(j+1)|w^{(\alpha)}_j-1|=\lim\limits_{j\rightarrow\infty}(j+1)\cdot\frac{1-\frac{j+1}{j+2\alpha+1}}{1+\sqrt{\frac{j+1}{j+2\alpha+1}}}=\alpha,
$$
every weighted Bergman space $A_{\alpha}^2$ satisfies the polynomial growth condition.

\begin{lemma}\label{unbounded}
Suppose that $H^2_{\beta}$ is the weighted Hardy space induced by a weight sequence $w=\{w_k\}_{k=1}^{\infty}$ with $w_k\rightarrow 1$.  Let $\varphi(z)=\frac{t-z}{1-tz}$, where $t\in (0,1)$, be a M\"{o}bius transformation on $\mathbb{D}$. Denote $\mathfrak{F}=\{\frac{\varphi^n(z)}{1-tz}\}_{n=0}^{\infty}$ and $\mathfrak{F}_{\beta}=\{\frac{1}{1-tz}\frac{\varphi^n(z)}{\beta_n}\}_{n=0}^{\infty}$.
If
\[
\lim\limits_{k\rightarrow\infty}(k+1)(w_k-1)=-\infty,
\]	
then $\mathfrak{F}_{\beta}$ is a Riesz base of $H^2_{\beta}$.
\end{lemma}
\begin{proof}
Firstly, we will show the following Claim.

{\bf Claim.} For any $N\in\mathbb{N}$, there exists $M_N>0$ such that for every $n\geq M_N$,
\[
\frac{\|\varphi^n(z)\|_{A_{N}^2}}{\|z^n\|_{A_{N}^2}}\geq \frac{(1+t)^{N}}{2(1-t)^{N-1}\sqrt{2\pi(2N-1)}}.
\]

Notice that
\begin{align*}
\|(\varphi')^{N}(z)\|^2_{H^2}
&=\frac{(1-t^2)^{2N}}{2\pi}\int_{0}^{2\pi}\frac{1}{(1-t{\textrm e}^{{\textbf i}\theta})^{2N}(1-t{\textrm e}^{-{\textbf i}\theta})^{2N}} d\theta \\
&=\frac{(1-t^2)^{2N}}{2\pi}\int_{0}^{2\pi}\frac{1}{(1+t^2-2t\cos\theta)^{2N}} d\theta \\
&\geq\frac{(1-t^2)^{2N}}{2\pi}\cdot 2\cdot \int_{0}^{\pi}\frac{\sin\theta}{(1+t^2-2t\cos\theta)^{2N}} d\theta \\
&=\frac{(1-t^2)^{2N}}{2\pi}\cdot2 \cdot\frac{1}{2t(2N-1)}\cdot(\frac{1}{(1-t)^{2(2N-1)}}-\frac{1}{(1+t)^{2(2N-1)}}) \\
&\geq\frac{(1-t^2)^{2N}}{2\pi}\cdot\frac{1}{t(2N-1)}\cdot\frac{t}{(1-t)^{2(2N-1)}} \\
&=\frac{(1+t)^{2N}}{2\pi(2N-1)(1-t)^{2N-2}}.
\end{align*}
Then, as $N\rightarrow+\infty$,
\[
\|(\varphi')^{N}(z)\|_{H^2} \geq \frac{(1+t)^{N}}{(1-t)^{N-1}\sqrt{2\pi(2N-1)}}\rightarrow +\infty.
\]

Following from \cite{KMis12}, one can see the unitary representation $D_{N}^+$ for the analytic automorphism group $\textrm{Aut}(\mathbb{D})$ on the weighted Bergman space $A_{N}^2$,
\[
D_{N}^+(g^{-1})(f)=(g')^{N}\cdot(f\circ g), \ \ f\in A_{N}^2, \ g\in \widetilde{\textrm{Aut}(\mathbb{D})},
\]
where $\widetilde{\textrm{Aut}(\mathbb{D})}$ is the universal covering group of $\textrm{Aut}(\mathbb{D})$.

Then, we have
\[
\|\varphi^n(z)\|_{A_{N}^2}=\|(\varphi'(z))^{N}\cdot(\varphi^n\circ \varphi(z))\|_{A_{N}^2}=\|z^n(\varphi'(z))^{N}\|_{A_{N}^2}.
\]
Denote $(\varphi'(z))^{N}=\sum_{k=0}^{\infty}c^{(N)}_kz^k$. There exists $K\in\mathbb{N}$ such that
\[
\sum\limits_{k=0}^{K}|c^{(N)}_k|^2\geq\frac{1}{2}\sum\limits_{k=0}^{\infty}|c^{(N)}_k|^2=\frac{1}{2}\|(\varphi')^{N}(z)\|^2_{H^2}.
\]
Since $w^{(N)}_k\rightarrow 1$, there exists $M_N>0$ such that for every $n\geq M_N$,
\[
\frac{\beta_{k+n}^{(N)}}{\beta_{n}^{(N)}}= \prod\limits_{i=n+1}^{n+k}w^{(N)}_{i} \geq \frac{\sqrt{2}}{2}, \ \ \ \text{for} \ k=1,2,\ldots,K.
\]
Then,
\[
\frac{\|\varphi^n(z)\|_{A_{N}^2}}{\|z^n\|_{A_{N}^2}}=\frac{\|z^n(\varphi'(z))^{N}\|_{A_{N}^2}}{\beta_n^{(N)}}
\geq \sqrt{\sum\limits_{k=0}^{K}|c^{(N)}_k|^2\left(\frac{\beta_{k+n}^{(N)}}{\beta_{n}^{(N)}}\right)^2}
\geq \frac{1}{2}\|(\varphi')^{N}\|_{H^2}.
\]
This completes the proof of the Claim.

Following from $\lim\limits_{k\rightarrow\infty}(k+1)(w_k-1)=-\infty$, we may assume all $w_k<1$ without loss of generality and we could obtain the followings.
\begin{enumerate}
\item [(I)]    For any $N\in \mathbb{N}$, there exists a positive integer $i_N$ such that for any $i\geq i_N$,
$w_i\leq w_i^{(N)}$.
\item [(II)]  Furthermore,
$\lim\limits_{k\rightarrow\infty} \beta_n/\beta_n^{(N)}=0$ and then
there exists a positive integer $n_N>i_N$ such that for any $n\geq n_N$, $\beta_n/\beta_n^{(N)}\leq \beta_{i_N}$.
\end{enumerate}

Consequently, for any $n\geq n_N$, the followings hold.
\begin{enumerate}
\item [(1)]
If $0\leq k\leq i_N$, we have
\[
\frac{\beta_{k}}{\beta_n}\geq \frac{\beta_{k}}{\beta_{i_N}\beta_{n}^{(N)}}\geq \frac{1}{\beta_{n}^{(N)}}  \geq\frac{\beta_k^{(N)}}{\beta_{n}^{(N)}}.
\]
\item [(2)] for every $i_N< k\leq n$, we have
\[
\frac{\beta_{k}}{\beta_n}= \prod\limits_{i=k+1}^{n}\frac{1}{w_{i}}\geq \prod\limits_{i=k+1}^{n}\frac{1}{w_i^{(N)}} =\frac{\beta_k^{(N)}}{\beta_{n}^{(N)}}.
\]
\end{enumerate}
Then, for any $n\geq n_N$ and any $0\leq k\leq n$, we have
\[
\frac{\beta_{k}}{\beta_n}\geq\frac{\beta_k^{(N)}}{\beta_{n}^{(N)}}.
\]
Therefore, for any $n\geq n_N$,
\begin{align*}
\frac{\|\varphi^{n}(z)\|^2_{H^2_{\beta}}}{\|z^{n}\|^2_{H^2_{\beta}}}&=\sum\limits_{k=0}^{\infty}|\widehat{\varphi^{n}}(k)|^2\cdot \frac{\beta_{k}^2}{\beta_n^2}  \\
&\geq  \sum\limits_{k=0}^{n}|\widehat{\varphi^{n}}(k)|^2\cdot \left(\frac{\beta_k^{(N)}}{\beta_{n}^{(N)}}\right)^2  \\
&\geq  \sum\limits_{k=0}^{\infty}|\widehat{\varphi^{n}}(k)|^2\cdot \left(\frac{\beta_k^{(N)}}{\beta_{n}^{(N)}}\right)^2 - \sum\limits_{k=n+1}^{\infty}|\widehat{\varphi^{n}}(k)|^2 \\
&\geq \frac{\|\varphi^{n}(z)\|^2_{A_{N}^2}}{\|z^{n}\|^2_{A_{N}^2}}-1.
\end{align*}
Together with the above Claim, one can see that the sequence $\{\frac{\varphi^n(z)}{\beta_n}\}_{n=0}^{\infty}$ is not uniformly bounded. Since $M_{\frac{1}{1-tz}}$ is a bounded invertible linear operator on $H^2_{\beta}$, the sequence
$\mathfrak{F}_{\beta}=\{\frac{1}{1-tz}\frac{\varphi^n(z)}{\beta_n}\}_{n=0}^{\infty}$ is not a Riesz base of $H^2_{\beta}$.
\end{proof}

\begin{theorem}\label{nosim}
Let $H^2_{\beta}$ be the weighted Hardy space induced by a weight sequence $w=\{w_k\}_{k=1}^{\infty}$ with $w_k\rightarrow 1$. Let $\varphi(z)=\frac{t-z}{1-tz}$, $t\in(0,1)$, be a M\"{o}bius transformation on $\mathbb{D}$.
If $M_z$ is strictly cyclic on $H^2_{\beta}$ and
\[
\lim\limits_{k\rightarrow\infty}(k+1)(w_k-1)=+\infty,
\]	
then the Hermitian holomorphic vector bundle $E_{S_{\beta^{-1}}}(\mathbb{D})$ is not similar to the vector bundle $E_{\varphi(S_{\beta^{-1}})}(\mathbb{D})$.
\end{theorem}
\begin{proof}
Suppose that the Hermitian holomorphic vector bundle $E_{S_{\beta^{-1}}}(\mathbb{D})$ is similar to the vector bundle $E_{\varphi(S_{\beta^{-1}})}(\mathbb{D})$. By Lemma \ref{cs-cs}, the multiplication operators $M_z$ and $M_{\varphi}$ are similar on $H^2_{\beta}$. Denote $\mathfrak{F}=\{\frac{\varphi^n(z)}{1-\overline{z_0}z}\}_{n=0}^{\infty}$, $\mathfrak{F}_{\beta}=\{\frac{1}{1-\overline{z_0}z}\frac{\varphi^n(z)}{\beta_n}\}_{n=0}^{\infty}$ and $\mathfrak{F}_{\beta^{-1}}=\{\frac{\beta_n\varphi^n(z)}{1-\overline{z_0}z}\}_{n=0}^{\infty}$. By Lemma \ref{similaroperator}, one can see that
$\mathfrak{F}_{\beta}$ is a Riesz base of $H^2_{\beta}$. Furthermore, by Lemma \ref{CorA}, $\mathfrak{F}_{\beta^{-1}}$ is also a Riesz base of $H^2_{\beta^{-1}}$. However, this is a contradiction to Lemma \ref{unbounded} since $\lim\limits_{k\rightarrow\infty}(k+1)({1}/{w_k}-1)=-\infty$. This finishes the proof.
\end{proof}

\begin{lemma}[Corollary 3 in \cite{Shi}, p~97]\label{strictcyclic}
If $M_z$ is strictly cyclic on $H^2_{\widetilde{\beta}}$ with weight sequence $\widetilde{w}=\{\widetilde{w}_k\}_{k=1}^{\infty}$, and if $\{a_k\}_{k=1}^{\infty}$ is a decreasing sequence of positive numbers, then $M_z$ is strictly cyclic on $H^2_{\beta}$ with weight sequence $\{a_k\widetilde{w}_k\}_{k=1}^{\infty}$.
\end{lemma}

\begin{example}\label{nln}
Suppose that $\varphi(z)=\frac{t-z}{1-tz}$, $t\in(0,1)$, is a M\"{o}bius transformation on $\mathbb{D}$. Let $\widetilde{w}_k=\frac{k+2}{k+1}$, $a_k={\textrm e}^{\ln^2(k+3)-\ln^2(k+2)}$ and $w_k=a_k\widetilde{w}_k$ for all $k=1,2,\ldots$. Let $H^2_{\widetilde{\beta}}$ and $H^2_{\beta}$ be the weighted Hardy spaces with weight sequences $\widetilde{w}=\{\widetilde{w}_k\}_{k=1}^{\infty}$ and $w=\{w_k\}_{k=1}^{\infty}$, respectively.

It is easy to see that
\begin{align*}
\lim\limits_{k\rightarrow\infty}\ln^2(k+3)-\ln^2(k+2)&=\lim\limits_{k\rightarrow\infty}\ln(1+\frac{1}{k+2})\cdot (\ln (k+3)+\ln (k+2)) \\
&=\lim\limits_{k\rightarrow\infty}\frac{\ln (k+3)+\ln(k+2)}{k+2} \\
&=0.
\end{align*}
Then, $w_k=\frac{k+2}{k+1}\cdot \textrm{e}^{\ln^2(k+3)-\ln^2(k+2)}\rightarrow 1$ as $k\rightarrow \infty$. Moreover,
\begin{align*}
&\lim\limits_{k\rightarrow\infty}(k+1)(w_k-1) \\
=&\lim\limits_{k\rightarrow\infty}(k+1)\left(\frac{k+2}{k+1}\cdot \textrm{e}^{\ln^2(k+3)-\ln^2(k+2)}-1\right) \\
=&\lim\limits_{k\rightarrow\infty}(k+1)\left(\frac{k+2}{k+1}\cdot \textrm{e}^{\ln^2(k+3)-\ln^2(k+2)}-\frac{k+2}{k+1}\right)+(k+1)\left(\frac{k+2}{k+1}-1\right) \\
=&\lim\limits_{k\rightarrow\infty}(k+2)\cdot(\ln^2(k+3)-\ln^2(k+2))+1 \\
=&\lim\limits_{k\rightarrow\infty}(k+2)\cdot\ln(1+\frac{1}{k+2})\cdot (\ln (k+3)+\ln (k+2))+1  \\
=&\lim\limits_{k\rightarrow\infty} \ln (k+3)+\ln (k+2)+1 \\
=&+\infty.
\end{align*}
Notice that $\left(\ln^2 x\right)'=\frac{2(1-\ln x)}{x^2}\leq0$ for $x\geq \textrm{e}$. The function $\ln^2 x$ is concave on $[\textrm{e}, +\infty)$. Then for $k\geq 1$
\[
\ln^2(k+4)-\ln^2 (k+3)\leq \ln^2(k+3)-\ln^2(k+2),
\]
which implies $a_k={\rm e}^{\ln^2(k+3)-\ln^2(k+2)}$ is decreasing with respect to $k$. In addition, $M_z$ is strictly cyclic on $H^2_{\widetilde{\beta}}$ (Example 1 in \cite{Shi}, p~99). Then, by Lemma \ref{strictcyclic}, $M_z$ is strictly cyclic on the weighted Hardy space $H^2_{\beta}$.

Therefore, by Theorem \ref{nosim}, the Hermitian holomorphic vector bundle $E_{S_{\beta^{-1}}}(\mathbb{D})$ is not similar to the vector bundle $E_{\varphi(S_{\beta^{-1}})}(\mathbb{D})$.
\end{example}

\begin{remark}
So far, one can see that each automorphism $\varphi$ induces a similar deformation from the Hermitian holomorphic vector bundle $E_{S_{\beta}}(\mathbb{D})$ to the vector bundle $E_{\varphi(S_{\beta})}(\mathbb{D})$ for weighted Hardy spaces with polynomial growth, but it is not true for weighted Hardy spaces with intermediate growth. So the polynomial growth condition is indispensable in our main theory.
\end{remark}

Our main theorems and related results could also answer some problems proposed in operator theory and K-theory. One problem was proposed by R. Douglas \cite{D07} in 2007.
\begin{question}[\cite{{D07}}]\label{Dconj}
Is $M_B$ on $L^2_a(D)$ similar to $M_z\otimes \mathbf{I}_m$ on $L^2_a(D)\otimes\mathbb{C}^m$, where $L^2_a(D)$ is the Bergman space and $m$ is the multiplicity of the finite Blaschke product $B(z)$?
\end{question}
Notice that Corollary \ref{Blaschesimilar} is a positive answer to the geometric bundle version of Douglas's problem for weighted Hardy spaces of polynomial growth. Then, together with Lemma \ref{cs-cs} and Theorem \ref{nosim}, we could also answer the generalized version of Douglas's problem.
\begin{corollary}\label{MB}
Let $B(z)$ be a finite Blaschke product with order $m$ on $\mathbb{D}$. Then $M_B\sim \bigoplus_{1}^{m}M_z$ on a weighted Hardy space $H^2_{\beta}$ of polynomial growth, and it is possibly false for weighted Hardy spaces with intermediate growth.
\end{corollary}

At the end, we give the $K_0$-group of the commutant algebra of a multiplication operator on a weighted Hardy space of polynomial growth, which is isomorphic to the $K_0$-group of some certain invariant deformation algebra.

\begin{corollary}\label{K0group}
Let $H^2_{\beta}$ be a weighted Hardy space of polynomial growth, and let $f(z)$ be an analytic function on $\overline{\mathbb{D}}$. Then
\[
K_0\left(\{M_f \}'_{H^2_{\beta}}\right)\cong K_0\left(\mathcal{L}_{inv}(E_{f(S_{\beta^{-1}})})\right)\cong \mathbb{Z}.
\]
\end{corollary}
\begin{proof}
By our Main Theorem, there exists a unique positive integer $m$ and an analytic function $h\in \textrm{Hol}(\overline{\mathbb{D}})$ inducing an indecomposable Hermitian holomorphic vector bundle $E_{h(S_{\beta})}$, such that
\[
E_{f(S_{\beta^{-1}})}\sim\bigoplus\limits_1^m E_{h(S_{\beta^{-1}})},
\]
Then,
\[
\{f(S_{\beta^{-1}})\}'_{H^2_{\beta^{-1}}}=\mathcal{L}_{inv}(E_{f(S_{\beta^{-1}})})\cong\mathcal{L}_{inv}\left(\bigoplus\limits_1^m E_{h(S_{\beta^{-1}})}\right)=
\left\{\bigoplus_{i=1}^{m}h(S_{\beta^{-1}})\right\}'_{\bigoplus_{i=1}^{m}H^2_{\beta^{-1}}}.
\]
Since $E_{h(S_{\beta^{-1}})}$ is an indecomposable Hermitian holomorphic vector bundle, $h(S_{\beta^{-1}})$ is a strongly irreducible Cowen-Douglas operator on $H^2_{\beta^{-1}}$. Then, following from Theorem 2.11 in \cite{J04}, we have
\[
K_0\left(\mathcal{L}_{inv}(E_{f(S_{\beta^{-1}})})\right) \cong K_0\left(\{f(S_{\beta^{-1}})\}'_{H^2_{\beta^{-1}}}\right) \cong K_0\left(\left\{\bigoplus_{i=1}^{m}h(S_{\beta^{-1}})\right\}'_{\bigoplus_{i=1}^{m}H^2_{\beta^{-1}}}\right)\cong\mathbb{Z}.
\]
By Lemma \ref{hFBS},
\[
\mathcal{L}_{inv}(E_{f(S_{\beta^{-1}})})
=T_{\beta, \beta^{-1}}\left\{Y^* ;~Y\in\{M_f\}'_{H^2_{\beta}} \right\}T_{\beta^{-1}, \beta}.
\]
Consequently,
\[
\{M^*_f \}'_{H^2_{\beta}}\cong \mathcal{L}_{inv}(E_{f(S_{\beta^{-1}})}).
\]
Therefore,
\[
K_0\left(\{M_h \}'_{H^2_{\beta}}\right)\cong K_0\left(\{M^*_h \}'_{H^2_{\beta}}\right)\cong K_0\left(\mathcal{L}_{inv}(E_{f(S_{\beta^{-1}})})\right) \cong\mathbb{Z}.
\]
\end{proof}

\section*{Declarations}

\noindent \textbf{Ethics approval}

\noindent Not applicable.

\noindent \textbf{Competing interests}

\noindent The author declares that there is no conflict of interest or competing interest.

\noindent \textbf{Authors' contributions}

\noindent All authors contributed equally to this work.

\noindent \textbf{Funding}

\noindent This work was supported by National Natural Science Foundation of China (Grant No. 12471120). The second author was partially supported by the Hebei Natural Science Foundation (Grant No. A2023205045).

\noindent \textbf{Availability of data and materials}

\noindent Data sharing is not applicable to this article as no data sets were generated or analyzed during the current study.

\end{document}